\documentclass[12pt,reqno] {amsart}%
\usepackage{amsfonts}
\usepackage{amsmath}
\usepackage{amssymb}
\usepackage{xypic,color}
\usepackage{graphicx}%
 \theoremstyle{plain}
\numberwithin{equation}{section}

\newtheorem{problem}{Problem}

\newtheorem{theorem}{Theorem}[section]
\newtheorem{conjecture}[theorem]{Conjecture}
\newtheorem{corollary}[theorem]{Corollary}
\newtheorem{definition}[theorem]{Definition}
\newtheorem{lemma}[theorem]{Lemma}
\newtheorem{proposition}[theorem]{Proposition}
\newtheorem{remark}[theorem]{Remark}

\def\bn{\begin{definition}}
\def\en{\end{definition}}
\def\ba{\begin{array}}
\def\ea{\end{array}}
\def\be{\begin{equation}}
\def\ee{\end{equation}}
\def\bd{\begin{description}}
\def\ed{\end{description}}
\def\bu{\begin{enumerate}}
\def\eu{\end{enumerate}}
\def\bi{\begin{itemize}}
\def\ei{\end{itemize}}

\newcommand{\e}[2][]{e^{#1}_{#2}} 

\def\ds{\displaystyle}

\def\i{\mathfrak{i}}
\allowdisplaybreaks
\def\<{\langle}
\def\>{\rangle}
\usepackage[colorlinks,
            linkcolor=black,
            anchorcolor=black,
            citecolor=black,
            ]{hyperref}

\makeatletter
\newtheorem*{rep@theorem}{\rep@title}
\newcommand{\newreptheorem}[2]{%
\newenvironment{rep#1}[1]{%
 \def\rep@title{#2 \ref{##1} (restated)}%
 \begin{rep@theorem}}%
 {\end{rep@theorem}}}
\makeatother
 \newreptheorem{theorem}{Theorem}
  \newreptheorem{proposition}{Proposition}
\begin{document}
\title[]
{ M\"{o}bius Disjointness for product flows of rigid dynamical systems and affine linear flows}
\author[]{Fei Wei}

\address{Yau Mathematical Sciences Center, Tsinghua University, Beijing 100084, China}

\email{weif@mail.tsinghua.edu.cn}
\maketitle
\begin{abstract}
We obtain that Sarnak's  M\"{o}bius Disjointness Conjecture holds for product flows between affine linear flows on compact abelian groups of zero topological entropy and a class of rigid dynamical systems. To prove this, we show an estimate for the average value of the product of the M\"obius function and any polynomial phase over short intervals and arithmetic progressions simultaneously. In addition, we prove that the logarithmically averaged M\"obius Disjointness Conjecture holds for the product flow between any affine linear flow on a compact abelian group of zero entropy and any rigid dynamical system. As an application, we show that the logarithmically averaged M\"obius Disjointness Conjecture holds for every Lipschitz continuous skew product dynamical system on $\mathbb{T}^2$ over a rotation of the circle.
\end{abstract}
\section{Introduction}\label{introduction}
Let $\mu(n)$ be the M\"{o}bius function,
that is,
$\mu(n)$ is $0$ when $n$ is not square-free (i.e., divisible by a nontrivial square),
and $(-1)^{r}$ when $n$ is the product of $r$ distinct primes. The M\"obius function plays an important role in number theory. For example, the Prime Number Theorem is known to be equivalent to $\sum_{n\leq N}\mu(n)=o(N)$.
The Riemann hypothesis holds if and only if $\sum_{n\leq N} \mu(n)= o(N^{\frac{1}{2}+ \epsilon})$ for every $\epsilon>0$.
It is widely believed that certain randomness exists in the values of the M\"obius function, and this randomness predicts significant cancellations in the summation of $\mu(n)\xi(n)$ for any ``reasonable" sequence $\xi(n)$. This rather vague principle is known as an instance of the ``M\"obius randomness principle" (see e.g., \cite[Section 13.1]{IK}). In \cite{Sar}, Sarnak made this principle precise by identifying the notion of a ``reasonable" sequence as an arithmetic function realized in a topological dynamical system\footnote{Let $(X,T)$ be a topological dynamical system (or a flow), that is $X$ is a compact metric space and $T:X\rightarrow X$ a continuous map. For $f\in C(X)$ and $x_{0}\in X$, we say $n\mapsto f(T^{n}x_{0})$ \emph{an arithmetic function realized in $(X,T)$}.}  with zero (topological) entropy.
Precisely,
\begin{conjecture}[M\"{o}bius Disjointness Conjecture]\label{moebius disjointness conjecture}
Let $X$ be a compact Hausdorff space and $T$ a continuous map on $X$ with zero topological entropy, then
\begin{equation}\label{standard average}
\lim_{N\rightarrow \infty}\frac{1}{N}\sum_{n=1}^{N}\mu(n)f(T^{n}x_{0})=0
\end{equation}
for any $x_{0}\in X$ and $f\in C(X)$.
\end{conjecture}
This conjecture has been verified for a variety of dynamical systems. See \cite{JB,JPT,EJMT5,FJ,FH,HLW,HWWZ,HWY,HWZ,Kan,LW,JP,CM16,Vee,Wang,Wei2,xu}, to list a few. One such system closely related to this paper is the rigid (measure-preserving) dynamical system. The notion of (measure theoretical) rigidity was introduced by Furstenberg and Weiss in \cite{FurWei}. A measure-preserving dynamical system\footnote{Meaning that $(X,\mathcal{B},\nu)$ is a probability space and $T$ is a measurable,
measure-preserving transformation on $X$.} $(X,\mathcal{B},\nu,T)$ is called \emph{rigid} if there is a sequence $\{n_{j}\}_{j=1}^{\infty}$ of positive integers such that for any $f(x)\in L^2(X,\nu)$,
\[\lim_{j\rightarrow \infty}\|f\circ T^{n_{j}}-f\|_{L^2(\nu)}^2=0.\]
Rigid dynamical systems contain dynamical systems with discrete spectrum (see e.g., \cite{KP7}), almost all interval exchange transformations \cite[Corollary 3]{Cha} and a large class of skew products on the torus over a rotation of the circle \cite{Anzai2}.

The main motivation of this research on product flows is as follows. The M\"obius disjointness of product flows between zero entropy dynamical systems and rigid dynamical systems (which is implied by Conjecture \ref{moebius disjointness conjecture}) is of great interest in number theory. It is closely related to the distribution of the average value of the product of the M\"obius function and any deterministic sequence (a sequence realized in zero entropy dynamical systems) over short arithmetic progressions.  The former implies that for any fixed modulo and for almost all short arithmetic progressions, the average value of the product between the M\"obius function and any deterministic sequence is small. Such an arithmetic property is important and has been extensively studied in number theory, see e.g., \cite{KM} for the  M\"obius function, \cite{F,HW} for the product of the M\"obius function and nilsequence. Based on this indication, to solve the M\"obius disjointness of product flows of a zero entropy dynamical system and rigid dynamical systems, we need to explore more properties of the distribution, such as the dependence on the modulo. In this paper, we study the case where the zero entropy dynamical system is an affine linear flows on compact abelian groups of zero topological entropy, and investigate properties of the distribution in this case.

An affine linear transformation
of a compact abelian group $X$ means $x\mapsto Ax+b$, where $A$ is an automorphism of $X$ and $b\in X$. A known example of arithmetic functions realized from affine transformations on abelian groups is $e^{2\pi\i P(n)}$ for any $P(x)\in \mathbb{R}[x]$. Throughout the paper, for simplicity,
we use $(X,\nu,T)$ to denote a measure-preserving dynamical system. For a topological dynamical system $(X,T)$  and $x_{0}\in X$, we use $\mathcal{M}(x_{0};X,T)$ to denote the weak*
closure of $\{\frac{1}{N}\sum_{n=0}^{N-1}\delta_{T^{n}x_{0}}: N=1,2,\ldots\}$ in the space of Borel probability measures on $X$. It is known that $\mathcal{M}(x_{0};X,T)$ is non-empty and for any $\nu\in \mathcal{M}(x_{0};X,T)$, $\nu$ is a $T$-invariant Borel probability measure on $X$ (see e.g., \cite[Theorem 6.9]{wal}). Our first main result is the following theorem.
\begin{theorem}\label{disjointfromasymptoticallyperiodicfunctionsb}
Let $(X_{1},T_{1})$ be an affine linear flow on a compact abelian group of zero topological entropy. Let $(X_{2},T_{2})$ be a flow such that for any $x_{0}\in X_{2}$ and any $\nu\in \mathcal{M}(x_{0};X_{2},T_{2})$, $(X_{2},\nu,T_{2})$ is rigid with the property that there are sequences $\{h_{j}\}_{j=1}^{\infty}$ and $\{s_{j}\}_{j=1}^{\infty}$ (may depend on $\nu$) of positive integers satisfying
\begin{equation}\label{restriction1b}
\lim_{j\rightarrow \infty}\frac{\log\log h_{j}}{\log h_{j}}\frac{s_{j}}{\varphi(s_{j})}=0
\end{equation}
and for any $g(x)\in L^2(X_{2},\nu)$,
\begin{equation}\label{asymptotical periodicity1b}
\lim_{j\rightarrow \infty}\frac{1}{h_{j}}\sum_{l=1}^{h_{j}}\|g\circ T_{2}^{ls_{j}}-g\|_{L^2(\nu)}^2=0,
\end{equation}
where $\varphi(s)$ is the Euler totient function.
Then the M\"obius Disjointness Conjecture holds for the product flow $(X_{1}\times X_{2},T_{1}\times T_{2})$.
\end{theorem}
It is easy to see from condition (\ref{asymptotical periodicity1b}) that there is a sequence $\{n_{j}\}_{j=1}^{\infty}$ of positive integers such that for any $g(x)\in L^2(X_{2},\nu)$, $\lim_{j\rightarrow \infty}\|g\circ T_{2}^{n_{j}}-g\|_{L^2(\nu)}^2=0$. Conditions (\ref{restriction1b}) and (\ref{asymptotical periodicity1b}) are restrictions on the speed of rigidity for $(X_{2},\nu,T_{2})$. There are many rigid dynamical systems satisfying these conditions, such as measure-preserving dynamical systems with discrete spectrum \cite[Proposition 5.4]{Wei2}, and measure-preserving dynamical systems with BPV rigidity and PR rigidity\footnote{It was explained in \cite[Remark 8.1]{Wei2} that both BPV rigidity and PR rigidity are included in the scenario described by conditions (\ref{restriction1b}), (\ref{asymptotical periodicity1b}). Also it was explained in \cite[Remark 8.4]{Wei2} that there is a measure-preserving dynamical system with discrete spectrum, but not satisfying BPV rigidity and PR rigidity.}, which was defined and studied in \cite{Kan}. In Lemma \ref{disjointfromasymptoticallyperiodicfunctions1b}, we shall prove a more general version of Theorem \ref{disjointfromasymptoticallyperiodicfunctionsb}.

Our second main result is about the estimate on the average value of the product of the M\"obius function and any polynomial phase over short intervals and arithmetic progressions simultaneously. It is the main ingredient of proving Theorem \ref{disjointfromasymptoticallyperiodicfunctionsb}. Below, we adopt the standard notation $e(\alpha):=e^{2\pi \i\alpha}$ for $\alpha\in \mathbb{R}$.
\begin{theorem}\label{the theorem we should prove}
Let $s\geq 1$ and $h\geq 3$ be integers. Let $d\geq 0$ be an integer. Suppose $P(x)\in \mathbb{R}[x]$ is of degree $d$. Then we have
\begin{equation}\label{an explict estimate of s and h}
\limsup_{N\rightarrow \infty}\frac{1}{Ns}\sum_{a=1}^s \sum_{n=1}^{N}\left|\frac{1}{h}\sum_{\substack{m=n+1\\m\equiv a(mod~s)}}^{n+hs}\mu(m)e(P(m))\right|\ll\frac{s}{\varphi(s)}\frac{\log\log h}{\log h},
\end{equation}
where the implied constant at most depends on $d$.
\end{theorem}
We list some results related to formula (\ref{an explict estimate of s and h}). For $s=1$, (\ref{an explict estimate of s and h}) becomes
\begin{equation}\label{the average of mobius in short interval}
\limsup_{N\rightarrow \infty}\frac{1}{N}\sum_{n=1}^{N}\left|\frac{1}{h}\sum_{m=n+1}^{n+h}\mu(m)e(P(m))\right|\ll \frac{\log\log h}{\log h}.
\end{equation}
As a consequence of the results of  Matom\"{a}ki-Radziwi{\l}{\l} \cite{KM} and  Matom\"{a}ki-Radziwi{\l}{\l}-Tao \cite{KMT17}, the above estimate holds for $P(x)=e(\alpha x)$, $\alpha\in \mathbb{R}$. In \cite{HW}, He-Wang further extended the result to any $P(x)\in \mathbb{R}[x]$, where the implied constant at most depends on the degree of $P(x)$.
Moreover, they proved that $e(P(n))$  can be replaced by any nilsequence $n \mapsto F(g^{n}x_{0})$.  Here $x_{0}\in G/\Gamma$ (a nilmanifold), $g\in G$ and $F:G/\Gamma\rightarrow \mathbb{C}$ is a Lipschitz function. See also \cite{EALdlR}, \cite{FFKPL} for earlier partial results in this direction.

For general $s$, Kanigowski, Lema{\'{n}}czyk and Radziwi{\l}{\l} \cite{Kan} showed that when $P(x)$ is a constant (i.e., $d=0$), then for any $\varepsilon>0$ and $h,s$ large enough (at most depending on $\epsilon$), the left-hand side of formula (\ref{an explict estimate of s and h})$\leq \epsilon$ whenever $\sum_{p|s}1/p\leq (1-\epsilon)\sum_{p\leq h}1/p$. In Theorem \ref{the theorem we should prove}, we showed that a similar estimate holds for all polynomials $P(x)$. Note that $\frac{s}{\varphi(s)}=\prod_{p|s}(1-1/p)^{-1}$ and the estimate $\sum_{p\leq h}1/p=\log\log h+O(1)$. So the above mentioned result in \cite{Kan} is contained in Theorem \ref{the theorem we should prove}.
Moreover, we study a general version of Theorem \ref{the theorem we should prove} replacing $\mu$ by any bounded multiplicative function. This will be given in Theorem \ref{the estimate of general multiplicative functions}.

We below explain the basic idea of proving Theorem \ref{the theorem we should prove}. The proof is inspired by Green-Tao's work on the average of the M\"obius function and nilsequences over long intervals \cite{GT2, GT}, and He-Wang's work on the average of the M\"obius function and nilsequences over short intervals \cite{HW}. The main technique they used is different versions of quantitative factorization of finite polynomials in nilmanifolds. For our situation, we should estimate the average of the M\"obius function and polynomial phases over short arithmetic progressions. Firstly, we give an appropriate quantitative factorization of finite polynomials in abelian nilmanifolds. Secondly, we use the nilpotent ``major and minor arcs" method, introduced in \cite{GT}. In the major arc part, we use the result \cite[Theorem 1.6]{KMT} on sums of multiplicative functions over almost all short arithmetic progressions. In the minor arc case we use a variant of the arguments of He-Wang \cite[Section 6]{HW} to obtain the cancellation we need.

It is still open whether the M\"obius Disjointness Conjecture holds for all rigid dynamical systems or the product flows between them and affine linear flows on compact abelian groups of zero topological entropy. More precisely,

\begin{problem}\label{sarnak's conjecture for rigid dynamical systems}
Let $(X_{1},T_{1})$ be an affine linear flow on a compact abelian group of zero topological entropy. Let $(X_{2},T_{2})$ be a flow such that for any $\nu$, a $T_{2}$-invariant Borel probability measure on $X_{2}$, $(X_{2},\nu, T_{2})$ is rigid. Let $T=T_{1}\times T_{2}$. Is it true that for any $F\in C(X_{1}\times X_{2})$ and $x\in X_{1}\times X_{2}$,
\begin{equation}\label{720formula1}
\lim_{N\rightarrow \infty}\frac{1}{N}\sum_{n=1}^{N}\mu(n)F(T^{n}x)=0.
\end{equation}
\end{problem}
Problem \ref{sarnak's conjecture for rigid dynamical systems} is closely related to the coefficient of $\frac{\log\log h}{\log h}$ in the right-hand side of (\ref{an explict estimate of s and h}). We explain the connection as follows. Usually the estimate of the left-hand side of (\ref{an explict estimate of s and h}) equals $o_{s}(h)$. We hope $o_s(h) = o(h)$, that is, the little ``o" term is independent of $s\geq 1$. This is related to the Chowla conjecture \cite{Ch} (Remark \ref{related to Chowla's conjecture}) and gives a sufficient condition to Problem \ref{sarnak's conjecture for rigid dynamical systems}. So we need to make explicitly the dependence of $o_s$ on $s$.
From Theorem \ref{the theorem we should prove}, we know that there is a function $\phi$ such that $o_s(h)=\phi(s)\frac{\log\log h}{\log h}$. So if we restrict the decay rate of the rigidity
\[\|g\circ T_{2}^{s_{j}}-g\|_{L^2(\nu)}^2\]
of $(X_{2},\nu, T_{2})$ to $O\Big(\exp\big(-(\phi(s_{j}))^{A}\big)\Big)$ for a sequence $\{s_{j}\}_{j=1}^{\infty}$ with $\lim_{j\rightarrow \infty}\phi(s_{j})=\infty$ and a given $A>1$, then we have the M\"obius disjointness for $(X_{1}\times X_{2},T_{1}\times T_{2})$.

We have a positive answer to Problem \ref{sarnak's conjecture for rigid dynamical systems} under the logarithmic average, namely we restrict formula (\ref{720formula1}) to logarithmic average $\frac{1}{\log N}\sum_{n=1}^{N}$ rather than the standard Ces\`{a}ro average $\frac{1}{N}\sum_{n=1}^{N}$.
\begin{theorem}\label{0707theorem1}
Make the same assumptions as Problem \ref{sarnak's conjecture for rigid dynamical systems}. Then the logarithmically averaged M\"obius Disjointness Conjecture holds for the product flow $(X_{1}\times X_{2},T_{1}\times T_{2})$. That is, for any $F\in C(X_{1}\times X_{2})$ and $x\in X_{1}\times X_{2}$,
\begin{equation}\label{equation for Proposition 1.8a}
\lim_{N\rightarrow \infty}\frac{1}{\log N}\sum_{n=1}^{N}\frac{\mu(n)F(T^{n}x)}{n}=0.
\end{equation}
where $T=T_{1}\times T_{2}$.
\end{theorem}

The major ingredients of the proof of Theorem \ref{0707theorem1} are the method used to prove Theorem \ref{disjointfromasymptoticallyperiodicfunctionsb}, Tao's result \cite{Tao16}, and Frantzikinakis-Host's result \cite{FH} on logarithmic average of the self-correlations of the M\"obius function twisted by polynomial phases (see equation (\ref{215formula41})).
We also refer to \cite{HXY} for other results on logarithmically averaged M\"obius Disjointness Conjecture.

It was considered in \cite{Fav,HWY,Kan,KL,JP,Wang} that the M\"obius disjointness for the skew product dynamical systems $(\mathbb{T}^2, T_{\alpha,\psi} )$, where $\mathbb{T}^{1}=\mathbb{R}/\mathbb{Z}$, $\mathbb{T}^2=\mathbb{T}^{1}\times\mathbb{T}^{1}$, $\alpha\in [0,1)$ and
\begin{equation}\label{formula718}
T_{\alpha, \psi}(x,y)=(x+\alpha, y+\psi(x))
\end{equation}
with $\psi$ a continuous map from $\mathbb{T}^{1}$ to itself. So far, it is known that the M\"obius Disjointness Conjecture holds for $(\mathbb{T}^2, T_{\alpha,\psi})$ with $\psi$ of class $C^{1+\epsilon}$. That is for any $\alpha\in [0,1)$ and $\psi$ of class $C^{1+\epsilon}$,
\[
\lim_{N\rightarrow \infty}\frac{1}{N}\sum_{n=1}^{N}\mu(n)f(T^{n}_{\alpha,\psi}(x_{0},y_{0}))=0,~~~\forall (x_{0},y_{0})\in \mathbb{T}^2,~~~\forall f\in C(\mathbb{T}^2).
\]
It is not known whether the M\"obius disjointness of $(\mathbb{T}^2, T_{\alpha,\psi})$ holds for any $\psi$ of Lipschitz continuous. As a corollary of Theorem \ref{0707theorem1}, we obtain that the logarithmically averaged M\"obius Disjointness Conjecture holds for this class of skew products.

\begin{corollary}\label{0707corollary1}
Suppose that $\alpha\in [0,1)$, $\psi:\mathbb{T}^{1}\rightarrow \mathbb{T}^{1}$ is Lipschitz continuous and $T_{\alpha, \psi}$ is in (\ref{formula718}). Then the logarithmically averaged M\"obius Disjointness Conjecture holds for product flows between $(\mathbb{T}^2, T_{\alpha,\psi})$ and affine linear flows on compact abelian groups of zero topological entropy. In particular, for any $P(x)\in \mathbb{R}[x]$, $(x_{0},y_{0})\in \mathbb{T}^2$ and $f\in C(\mathbb{T}^2)$,
\begin{equation}\label{0707equation1}
\lim_{N\rightarrow \infty}\frac{1}{\log N}\sum_{n=1}^{N}\frac{\mu(n)e(P(n))f(T^{n}_{\alpha,\psi}(x_{0},y_{0}))}{n}=0.
\end{equation}
\end{corollary}

%

\textbf{Notation.}
We call $e(f(n))$ a polynomial phase when $f(x)\in \mathbb{R}[x]$ is a polynomial. We use $1_{S}$ to denote the indicator of a predicate $S$,
that is $1_{S}=1$ when $S$ is true and $1_{S}=0$ when $S$ is false.
We also denote $1_{A}(n)=1_{n\in A}$ for any subset $A$ of $\mathbb{N}$.
For any finite set $C$,
$|C|$ denotes the cardinality of $C$.

We use $\{x\}$ and $\lfloor x \rfloor$ to denote the fractional part and the integer part of a real number $x$,
respectively.
We use $\|x\|_{\mathbb{R}/\mathbb{Z}}$ to denote the distance between $x$ and the set $\mathbb{Z}$,
i.e., $\|x\|_{\mathbb{R}/\mathbb{Z}}= \min(\{x\},1-\{x\})$.
For ease of notation,
we sometimes drop the subscript and write simply $\|x\|$. For $y\in \mathbb{R}^{m}$, $|y|$ is the $l^{\infty}$-norm of $y$.
For integer $N\geq 1$, we use $[N]$ to denote the set $\{1,\ldots,N\}$.

For two arithmetic functions $f(n)$ and $g(n)$, $f=o(g)$ means $\lim_{n\rightarrow \infty}f(n)/g(n)=0$; Let $g(n)$ be a positive functions of $n$,
$f\ll g$ means that there is an absolute constant $c$ such that $|f|\leq c g$;
$f=h+O(g)$ means $f-h\ll g$.

\textbf{Acknowledgments.} I would like to thank Kaisa Matom\"{a}ki for valuable discussions about the coefficient $(\log N)^{o(1)}$ in Theorem \ref{the estimate of general multiplicative functions}. I am very grateful for JinXin Xue for helpful suggestions on the manuscript. This paper is supported by the fellowship of China Postdoctoral Science Foundation 2020M670273.
\section{A quantitative factorization theorem}
In this section, we shall deduce a quantitative factorization for certain polynomials in two parameters. It is a variation of \cite[Theorems 1.19 and 10.2]{GT2} or \cite[Theorem 3.6]{HW} that is adapted to our situation.
The difference of these three versions lies in the intervals the parameters belong to. In \cite{GT2}, all parameters are in long intervals. In \cite{HW}, for polynomials $p(n, h)$, the first parameter $n$ varies in a long interval $[1,N]$, while $h$ varies in a short interval $[1,H]$ in the sense that $H$ can grow arbitrarily slower than $N$. In our case, $n$ is also allowed to vary in a long interval, while $h$ lies in a short arithmetic progression $\{h\in[1,Hs]:n+h\equiv a(mod~s)\}$. Here, $H,s$ are arbitrary positive integers, which can grow arbitrarily slower than $N$. Since the parameter $s$ can be much larger than $H$, it is not enough for our purpose if we just apply \cite{HW}  to $h=ls+a$. Our factorization theorem is as follows.
\begin{theorem}\label{decomposition theorem}
Let $m\geq 1$, $d\geq 0$ be integers. Let $f:\mathbb{Z} \rightarrow \mathbb{R}^{m}$ be a polynomial of degree $d$ and $g(n,h) = f(n+h)$. Let $R>2$. Let $s\geq 1$ be an integer. Then for any $B\geq 1, N,H\in \mathbb{N}$ such that $N>(Hs)^{O(1)}$ and $H> R^{O(1)}$, there is an integer $W\in [R, R^{O_{d}(B^{m})}]$, a connected subgroup $G'$ of $\mathbb{R}^{m}$, a set $\mathcal{N}\subseteq [N]$ with $|\mathcal{N}|\geq (1-W^{-B/2})N$, and a polynomial decomposition
\[g(n,h) =\mathcal{E}(n,h) + g'(n,h) + \gamma(n,h)\]
with $\mathcal{E}, g', \gamma: \mathbb{Z}^2\rightarrow \mathbb{R}^{m}$ satisfying

(\romannumeral1). $|\mathcal{E}(n,h+1) - \mathcal{E}(n,h)| \ll_d \frac{1}{sH/W}$ for $h\in [Hs]$.

(\romannumeral2). $g'(n,h)$ takes values in $G'$, and for any $n\in \mathcal{N}$, there are at least  $(1-W^{-B/2})s$ integers of $a\in[s]$ such that $\{g'(n,h)(mod~ \Gamma')\}_{h\in \mathcal{A}_{a}}$ is totally $W^{-B}$-equidistributed in $G'/\Gamma'$, where $\mathcal{A}_{a}=\{1\leq h\leq Hs:h\equiv a(mod~s)\}$ and $\Gamma'=G'\cap \mathbb{Z}^{m}$.

(\romannumeral3). There is an integer $q$ with $1\leq q \leq W$ such that $\{\gamma(n,h)(mod~\mathbb{Z}^{m})\}_{(n,h)\in \mathbb{Z}^2}$ is $qs$-periodic both in $n$ and $h$.
\end{theorem}
Theorem \ref{decomposition theorem} is fundamental in the proof of Theorem \ref{the estimate of general multiplicative functions}. Before proving it, we need some preparations and show some intermediate results. We first recall some definitions.
\begin{definition}$($\cite[Definition 2.7]{GT2}$)$
Let $d\geq 0$ and $f:\mathbb{Z} \rightarrow \mathbb{R}/\mathbb{Z}$ be a polynomial of degree $d$. Write $f(n)=\sum_{i=0}^d \alpha_i \binom{n}{i}$ with each $\alpha_{i}\in \mathbb{R}/\mathbb{Z}$. For any integer $N\geq 1$, the
\textbf{smoothness norm} of $f$ is defined by
\[
\|f\|_{C^\infty[N]}:= \max_{1\leq i \leq d} N^i \|\alpha_i\|_{\mathbb{R}/\mathbb{Z}}.
\]
\end{definition}

\begin{definition}$($\cite[Definition 1.2]{GT2}$)$
Let $G/\Gamma$ be a nilmanifold endowed with the unique normalized Haar measure. Let $\delta>0$ and $\mathcal{A}$ a finite arithmetic progression in $\mathbb{Z}$. A finite sequence $\{x_{n}\}_{n\in \mathcal{A}}$ in $G/\Gamma$ is said to be \textbf{$\delta$-equidistributed} in $G/\Gamma$ if
\[\Big|\mathbb{E}_{n\in \mathcal{A}}F(x_{n})-\int_{G/\Gamma}F\Big|\leq \delta\|F\|_{Lip}\]
for all Lipschitz function $F: G/\Gamma\rightarrow \mathbb{C}$, where
\[\mathbb{E}_{n\in \mathcal{A}}F(x_{n})=\frac{1}{|\mathcal{A}|}\sum_{n\in \mathcal{A}}F(x_{n})\] and
\begin{equation}\label{def of Lip}
\|F\|_{Lip}:=\sup_{x\in G/\Gamma}|F(x)|+\sup_{\substack{x,y\in G/\Gamma\\x\neq y}}\frac{|F(x)-F(y)|}{d_{G/\Gamma}(x,y)}.
\end{equation}
A finite sequence $\{x_{n}\}_{n\in \mathcal{A}}$ in $G/\Gamma$ is said to be totally \textbf{$\delta$-equidistributed} in $G/\Gamma$ if the sequence $\{x_{n}\}_{n\in \mathcal{A}'}$ is $\delta$-equidistributed in $G/\Gamma$ for all arithmetic progressions $\mathcal{A}'\subseteq \mathcal{A}$ of length at least $\delta N$.
\end{definition}
The next result gives a simple connection between the coefficients of a polynomial and its smoothness norm.
\begin{lemma}\label{coefficients}$($\cite[Lemma 3.2]{GT}$)$ Suppose that $f:\mathbb{Z}\rightarrow \mathbb{R}/\mathbb{Z}$ is a polynomial of the form $\sum_{i=0}^d \beta_i n^i$. Then there is a positive integer $D=O_d(1)$ such that $\|D\beta_i\|_{\mathbb{R}/\mathbb{Z}}\ll_d N^{-i} \|f\|_{C^\infty[N]}$ for all
$i=1,\ldots,d$.
\end{lemma}
The following is a ``strong recurrence'' result for polynomials $f:\mathbb{Z}\rightarrow \mathbb{R}$.
\begin{lemma}\label{CNnorm}$($\cite[Lemma 4.5]{GT2}$)$
Suppose that $f: \mathbb{Z}\rightarrow \mathbb{R}$ is a polynomial of degree $d$ with $d\geq 0$. Suppose that $N\geq 1$, $\delta \in(0,1/2)$ and $\epsilon \in(0,\delta/2]$. If $f(n)(mod~ \mathbb{Z})$ belongs to an interval $I\subseteq \mathbb{R}/\mathbb{Z}$ of length $\epsilon$ for at least $\delta N$ integers of $n\in[N]$. Then there is a $D\in \mathbb{Z}$ with $0<|D|\ll_{d} \delta^{-O_{d}(1)}$ such that
$\|Df(mod~\mathbb{Z})\|_{C^\infty[N]} \ll_d \epsilon \delta^{-O_d(1)}$.
\end{lemma}
The next one is a special case of the quantitative equidistribution result on polynomials $f:\mathbb{Z}\rightarrow \mathbb{R}^{m}$ (\cite[Theorem 2.9]{GT2} and \cite[Corollary 3.2]{HW}).
\begin{lemma}\label{Weyl rule}
Let $m\geq 1$, $d\geq 0$ and $N\geq 1$ be integers. Suppose that $f: \mathbb{Z}\rightarrow \mathbb{R}^{m}$ is a polynomial of degree $d$. Let $R>2$. If $\{f(n)(mod~\mathbb{Z}^{m})\}_{n\in [N]}$ is not totally $R^{-1}$-equidistributed in $\mathbb{R}^{m}/\mathbb{Z}^{m}$, then there is a $D\in \mathbb{Z}^{m}$ with $0<|D|\leq R^{O_{m,d}(1)}$ such that $\|D \cdot f(mod~\mathbb{Z})\|_{C^\infty[N]}\leq R^{O_{m,d}(1)}$.
\end{lemma}
Now, we prove a variant version of the above result for the equidistribution in arithmetic progressions that is applicable to our case. In the following, constants in the asymptotic notation $\ll $ and $O(\cdot)$ are allowed to depend on $m$ and $d$. These constants are not necessarily the same in each occurrence.

\begin{proposition}\label{0910proposition}
Let $m\geq 1$, $d\geq 0$, and let $g(n,h): \mathbb{Z}^2\rightarrow \mathbb{R}^{m}$ be a polynomial. Suppose that \[g(n,h)=\sum_{\substack{j,k\geq 0\\j+k\leq d}}\beta_{j,k}n^{j}h^{k}\] with each $\beta_{j,k}\in \mathbb{R}^{m}$. Let $\widetilde{R}>2$ and $N,H\in \mathbb{N}$ with $N,H\geq \widetilde{R}^{O(1)}$. Let $a,s$ be integers with $s\geq 1$ and $1\leq a\leq s$. Write $\mathcal{A}_{a}=\{h:1\leq h\leq Hs,h\equiv a(mod~s)\}$. Then at least one of the following holds:

(\romannumeral1). $\{g(n,h)(mod~\mathbb{Z}^{m})\}_{h\in \mathcal{A}_{a}}$ is totally $\widetilde{R}^{-1}$-equidistributed in $\mathbb{R}^{m}/\mathbb{Z}^{m}$ for all but at most $\widetilde{R}^{-1}Ns$ values of $(n,a)\in [N]\times [s]$.

(\romannumeral2). there is a $Q\in \mathbb{Z}^{m}$ with $0<|Q|\ll \widetilde{R}^{O(1)}$ such that for any $i=1,\ldots,d$ and $j=0,\ldots,d$ with $i+j\leq d$ and for any $k$ with $i\leq k\leq d-j$,
$
\|s^i Q\cdot \beta_{j,k}\|_{\mathbb{R}/\mathbb{Z}}\ll \widetilde{R}^{O(1)}N^{-j}H^{-i}s^{i-k}.
$
\end{proposition}

\begin{proof}
Suppose that
\begin{equation}\label{0910formula1}
g(n,ls+a)=\sum_{\substack{i,j\geq 0\\i+j\leq d}}w_{i,j}l^in^{j},
\end{equation}
where
\begin{equation}\label{0910formula2}
w_{i,j}=s^i\sum_{k=i}^{d-j}a^{k-i}\binom{k}{i}\beta_{j,k}.
\end{equation}
Assume claim $(\romannumeral1)$ does not hold.
Denote by $\mathcal{A}$ the set of all $(n,a)$ such that $\{g(n,h)(mod~\mathbb{Z}^{m})\}_{h\in \mathcal{A}_{a}}$ is not totally $\widetilde{R}^{-1}$-equidistributed in $\mathbb{R}^{m}/\mathbb{Z}^{m}$. Then $|\mathcal{A}|\geq \widetilde{R}^{-1}Ns$.  For each  $(n,a)\in \mathcal{A}$, by Lemma \ref{Weyl rule}, there is a $D\in \mathbb{Z}^{m}$ with $0<|D|\leq \widetilde{R}^{O(1)}$ such that
\be \label{eq10.24a}
\|D\cdot g(n,\cdot s+a)(mod~\mathbb{Z})\|_{C^\infty[H]}
\ll \widetilde{R}^{O(1)}.
\ee
By the pigeonhole principle, there is a common $D$ independent of $(n,a)$ with $0<|D| \leq \widetilde{R}^{O(1)}$ such that formula (\ref{eq10.24a}) holds for at least $\widetilde{R}^{-O(1)}Ns$ values of $(n,a)\in \mathcal{A}\subseteq [N]\times[s]$. We still denote the set of such $(n,a)$ by $\mathcal{A}$. By Lemma \ref{coefficients}, there is an integer $D_{1}=O(1)$ such that for any $i=1,\ldots,d$ and $(n,a)\in \mathcal{A}$,
\begin{equation}\label{0910formula3}
\|D_{1}\sum_{j=0}^{d-i}D\cdot w_{i,j}n^{j}\|_{\mathbb{R}/\mathbb{Z}}\ll \widetilde{R}^{O(1)}H^{-i}.
\end{equation}
By enlarging the constant implied in the constant $O(1)$ in the estimate of $|\mathcal{A}|\geq \widetilde{R}^{-O(1)}Ns$, there are at least $\widetilde{R}^{-O(1)}s$ choices of $a\in [s]$ such that for each $a$, there are at least $\widetilde{R}^{-O(1)}N$ choices of $n\in [N]$ with $(n,a)\in \mathcal{A}$. Denote the set of such $a$ by $\mathcal{B}$. Then $|\mathcal{B}|\geq \widetilde{R}^{-O(1)}s$. Treat $D_{1}\sum_{j=0}^{d-i}D\cdot w_{i,j}n^{j}$ as the polynomial of $n$. Then by Lemmas \ref{CNnorm} and \ref{Weyl rule} again, for each $a\in \mathcal{B}$, there is an integer $D_{2}$ with $0<|D_{2}|\ll \widetilde{R}^{O(1)}$, such that $\|D_{2}D_{1}D\cdot w_{i,j}\|_{\mathbb{R}/\mathbb{Z}}\ll \widetilde{R}^{O(1)}H^{-i}N^{-j}$ for $i=1,\ldots,d$ and $j=1,\ldots, d$ with $i+j\leq d$. We may assume that $D_{2}$ is independent of $a$, since this can be obtained after pigeonholding in the $\widetilde{R}^{O(1)}$ possible choices of $D_{2}$ and enlarging the constant $O(1)$ in the estimate of $|\mathcal{B}|$. This leads to
\begin{equation}\label{928}
\|D_{2}D_{1}\sum_{j=1}^{d-i}D\cdot w_{i,j}n^{j}\|_{\mathbb{R}/\mathbb{Z}}\ll \widetilde{R}^{O(1)}H^{-i}
\end{equation}
for $i=1,\ldots,d$. By (\ref{0910formula3}), (\ref{928}) and the triangle inequality, we obtain the diophantine information about $w_{i,0}$. Summarize the above analysis, together with (\ref{0910formula2}), we have that
\begin{equation}\label{0910formula4}
\|D_{2}D_{1}s^i\sum_{k=i}^{d-j}a^{k-i}\binom{k}{i}D\cdot \beta_{j,k}\|_{\mathbb{R}/\mathbb{Z}}\ll \widetilde{R}^{O(1)}N^{-j}H^{-i}
\end{equation}
holds for any $a\in \mathcal{B}$, $i=1,\ldots,d$ and $j=0,\ldots,d$ with $i+j\leq d$. For fixed $(i,j)\in \{1,\ldots,d\}\times \{0,\ldots,d\}$, we treat $D_{2}D_{1}s^i\sum_{k=i}^{d-j}a^{k-i}\binom{k}{i}D\cdot \beta_{j,k}$ as a polynomial of $a$. By an argument similar to that of (\ref{0910formula4}), there is an integer $D_{3}$ with $0<|D_{3}|\ll \widetilde{R}^{O(1)}$ such that for any $i=1,\ldots,d$ and $j=0,\ldots,d$ with $i+j\leq d$ and for any $i\leq k\leq d-j$,
\[
\|s^i D_{3}D_{2}D_{1}D\cdot \beta_{j,k}\|_{\mathbb{R}/\mathbb{Z}}\ll \widetilde{R}^{O(1)}N^{-j}H^{-i}s^{i-k}.
\]
Let $Q=D_{3}D_{2}D_{1}D$. We obtain claim (ii).
\end{proof}
\begin{corollary}\label{two dimensional weyl's rule}
Let $m\geq 1$, $d\geq 0$, and let $f(n)=\sum_{i=0}^{d}\alpha_{i}n^{i}$ with each $\alpha_{i}\in \mathbb{R}^{m}$. Suppose that $g(n,h)=f(n+h)$. Let $\widetilde{R}>2$ and $N,H\in \mathbb{N}$ with $N,H\geq \widetilde{R}^{O(1)}$. Let integers $s\geq 1$ and $1\leq a\leq s$. Write $\mathcal{A}_{a}=\{h:1\leq h\leq Hs,h\equiv a(mod~s)\}$. Then at least one of the following holds:

(\romannumeral1). $\{g(n,h)(mod~\mathbb{Z}^{m})\}_{h\in \mathcal{A}_{a}}$ is totally $\widetilde{R}^{-1}$-equidistributed in $\mathbb{R}^{m}/\mathbb{Z}^{m}$ for all but at most $\widetilde{R}^{-1}Ns$ values of $(n,a)\in [N]\times [s]$.

(\romannumeral2). there is a $Q\in \mathbb{Z}^{m}$ with $0<|Q|\ll \widetilde{R}^{O(1)}$ such that $\|sQ\cdot \alpha_l\|_{\mathbb{R}/\mathbb{Z}}
\leq \widetilde{R}^{O(1)} H^{-1}N^{-(l-1)}$ for $l=1,\ldots,d$.
\end{corollary}
\begin{proof}
Suppose that \[g(n,h)=\sum_{\substack{j,k\geq 0\\j+k\leq d}}\beta_{j,k}n^{j}h^{k},\]
where $\beta_{j,k}\in \mathbb{R}^{m}$.
By the assumption, $\beta_{j,k}=\alpha_{j+k}\binom{j+k}{j}$. Assume that $(\romannumeral1)$ does not hold. By Proposition \ref{0910proposition}, there is a $Q_{1}\in \mathbb{Z}^{m}$ with $0<|Q_{1}|\ll\widetilde{R}^{O(1)}$ such that for any $i=1,\ldots,d$ and $j=0,\ldots,d$ with $i+j\leq d$ and for any $k$ with $i\leq k\leq d-j$,
$
\|s^i Q_{1}\cdot \alpha_{j+k}\binom{j+k}{j}\|_{\mathbb{R}/\mathbb{Z}}\ll \widetilde{R}^{O(1)}N^{-j}H^{-i}s^{i-k}.
$
Let $i=k=1$, $j=l-1$. Let $Q=d! Q_{1}$. We obtain claim (\romannumeral2).
\end{proof}
Now using Corollary \ref{two dimensional weyl's rule}, we are ready to prove Theorem \ref{decomposition theorem}.
\begin{proof}[Proof of Theorem \ref{decomposition theorem}]
Suppose that $f(n)=\sum_{i=0}^{d}\alpha_{i}n^{i}$ with each $\alpha_{i}\in \mathbb{R}^{m}$. Let $\widetilde{R} = R^B$. If $\{g(n,h)(mod~\mathbb{Z}^{m})\}_{h\in \mathcal{A}_{a}}$ is totally $\widetilde{R}^{-1}$-equidistributed in $\mathbb{R}^{m}/\mathbb{Z}^{m}$ for all but at most $\widetilde{R}^{-1}Ns$ values of $(n,a)\in [N]\times [s]$,
then choose $\mathcal{E}(n,h) = \gamma(n,h) = 0, g'(n,h) = g(n,h)$ and $G'=\mathbb{R}^{m}$, $W=R$. Otherwise, applying Corollary \ref{two dimensional weyl's rule} to $g(n,h)$, then there is a $Q\in \mathbb{Z}^{m}$ with $0<|Q| \ll \widetilde{R}^{O(1)}$ such that for $j=1,\ldots,d$,
\[
\|sQ\cdot \alpha_j\|_{\mathbb{R}/\mathbb{Z}}
\leq \widetilde{R}^{O(1)} H^{-1}N^{-(j-1)}.
\]
We may choose $\beta_{j}\in \mathbb{R}^{m}$ with $Q\cdot \beta_{j}\in \mathbb{Z}/s$ and $|\alpha_{j}-\beta_{j}|<\widetilde{R}^{O(1)} H^{-1}N^{-(j-1)}/s$. We further choose $\tau_{j}\in \frac{ \mathbb{Z}^{m}}{qs}$ for some integer $q$ with with $0<|q|<\widetilde{R}^{O(1)}$ and $Q\cdot \beta_{j}=Q\cdot \tau_{j}$.  Let $\mathcal{E}_1(n,h)=\sum_{j=1}^{d}(\alpha_j-\beta_{j}) (n+h)^j+\alpha_{0}$ and $\gamma_{1}(n,h)=\sum_{j=1}^{d}\tau_{j}(n+h)^{j}$. Define
\[g_{1}(n,h)=g(n,h)-\mathcal{E}_1(n,h)-\gamma_{1}(n,h).\]
Then $Q\cdot g_{1}(n,h)=0$.
So $g_{1}(n,h)$ takes values in $G_{1}$, a connected proper subgroup $G_{1}$ of $G$, which is isomorphic (after a natural projection of coordinates, denoted by $\pi_{1}$) to $\mathbb{R}^{m-1}$. Additionally, using $N>(Hs)^{O(1)}$, one can verify that $|\mathcal{E}_{1}(n,h+1)-\mathcal{E}_{1}(n,h)|\ll_{d}\frac{\widetilde{R}^{O(1)}}{Hs}$ for any $h\in [Hs]$. Consider the lattice $\Gamma_{1}:=G_{1}\cap \mathbb{Z}^{m}$. Then it is not hard to check that $\Gamma_{1}$ is isomorphic to $A_{1}\mathbb{Z}^{m-1}$, where $A_{1}\in  GL_{m-1}(\mathbb{R})$ with integer entries of height $\widetilde{R}^{O(1)}$. Let $R_{1}=\widetilde{R}\cdot\widetilde{R}^{O(1)}=R^{O(B)}$ and $\widetilde{R}_{1}=R_{1}^{B}$. We now check weather $\{g_{1}(n,h)(mod~\Gamma_{1})\}_{h\in \mathcal{A}_{a}}$ is totally $\widetilde{R_{1}}^{-1}$-equidistributed in $G_{1}/\Gamma_{1}$ for all but at most $\widetilde{R}_{1}^{-1}Ns$ values of $(n,a)\in [N]\times [s]$. If it is so, we terminate the process and then choose $\mathcal{E}(n,h) = \mathcal{E}_{1}(n,h)$, $\gamma(n,h) =\gamma_{1}(n,h), g'(n,h) = g_{1}(n,h)$ and $G'=G_{1}$, $W=R_{1}$.

If it is not the case, we continue the process. Write $f_{1}(n)=\sum_{i=1}^{d}\alpha_{i,1}n^{i}$ such that $g_{1}(n,h)=f_{1}(n+h)$. Then $A_{1}^{-1}\circ \pi_{1}(f_{1}(n))$, $A_{1}^{-1}\circ \pi_{1}(g_{1}(n,h))$ are polynomials taking values in $\mathbb{R}^{m-1}$. We then again apply Corollary \ref{two dimensional weyl's rule} to $A_{1}^{-1}\circ \pi_{1}(g_{1}(n,h))$ with $m$ replaced by $m-1$ and $\widetilde{R}$ replaced by $\widetilde{R_{1}}$. Then by an argument similar to that of $g(n,h)$, we can find the corresponding $\widetilde{\mathcal{E}_{2}}(n,h)$, $\widetilde{\gamma_{2}}(n,h)$ and so define $\mathcal{E}_{2}(n,h)$, $\mathcal{\gamma}_{2}(n,h)$ by $\pi_{1}^{-1}\circ A_{1}(\widetilde{\mathcal{E}_{2}}(n,h))$ and by $\pi_{1}^{-1}\circ A_{1}(\widetilde{\mathcal{\gamma}_{2}}(n,h))$, respectively. Define $g_{2}(n,h)$ by $g_{1}(n,h)-\mathcal{E}_{2}(n,h)-\gamma_{2}(n,h)$. Assume that in the $l$-th step, we apply Corollary \ref{two dimensional weyl's rule} to $g_{l-1}(n,h)$ with $\widetilde{R}$ replaced by $\widetilde{R}_{l-1}$, where $\widetilde{R}_{l-1}=(R_{l-1})^{B}$ and  $R_{l-1}=\widetilde{R}_{l-2}\widetilde{R}_{l-2}^{O(1)}$. Then we obtain

$\bullet$ $\mathcal{E}_{l}(n,h)$ satisfies $|\mathcal{E}_{l}(n,h+1) - \mathcal{E}_{l}(n,h)| \ll_d (\widetilde{R}_{l-1})^{O(1)}/(Hs)$ for $h\in [Hs]$;

$\bullet$ $\{\gamma_{l}(n,h)(mod~\mathbb{Z}^{m})\}_{(n,h)\in \mathbb{Z}^2}$ is $q_{l}s$-periodic for some positive integer $q_{l}\leq (\widetilde{R}_{l-1})^{O(1)}$;

$\bullet$ $g(n,h)=\sum_{i=1}^{l}\mathcal{E}_{i}(n,h)+g_{l}(n,h)+\sum_{i=1}^{l}\gamma_{i}(n,h)$, where $g_{l}(n,h)$ is a polynomial taking values in $G_{l}$, and $\{\sum_{i=1}^{l-1}\gamma_{i}(n,h)(mod~\mathbb{Z}^{m})\}_{(n,h)\in \mathbb{Z}^2}$ is $(\prod_{i=1}^{l-1}q_{i})s$-periodic with $\prod_{i=1}^{l-1}q_{i}\leq \widetilde{R}_{l-1}$.

Since $m$ is a given non-negative integer, the total number of iterations is bounded by $m$. This implies that for some positive integer $l$,
$\{g_{l}(n,h)(mod~\Gamma_{l})\}_{h\in \mathcal{A}_{a}}$ is totally $(\widetilde{R}_{l})^{-1}$-equidistributed in $G_{l}/\Gamma_{l}$ for all but $(\widetilde{R}_{l})^{-1} Ns$ values of $(n,a)\in [N]\times [s]$, where $G_{l}$ is a connected proper subgroup of $\mathbb{R}^{m}$. Let $G'=G_{l}$ and $\Gamma'=\Gamma_{l}=G'\cap \mathbb{Z}^{m}$. Choose $W=R_{l}=R^{O(B^{m})}$, $\mathcal{E}(n,h)=\sum_{i=1}^{l}\mathcal{E}_{i}(n,h)$, $g'(n,h)=g_{l}(n,h)$ and $\gamma(n,h)=\sum_{i=1}^{l}\gamma_{i}(n,h)$. So we obtain that $\{g'(n,h)(mod~\Gamma')\}_{h\in \mathcal{A}_{a}}$ is totally $W^{-B}$-equidistributed in $G'/\Gamma'$ for all but at most $W^{-B}Ns$ values of $(n,a)\in [N]\times [s]$. Let $\mathcal{N}$ denote the set of $n\in [N]$ which satisfies that at least  $(1-W^{-B/2})s$ choices of $a\in[s]$ has the property that $\{g'(n,h)(mod~\Gamma')\}_{h\in \mathcal{A}_{a}}$ is totally $W^{-B}$-equidistributed in $G'/\Gamma'$. It is not hard to check that
\[(N-|\mathcal{N}|)W^{-B/2}s\leq W^{-B}Ns.\]
Then $|\mathcal{N}|\geq (1-W^{-B/2})N $. We complete the proof.
\end{proof}
\begin{remark}\label{form of g'}
{\rm From the above proof, it is not hard to see that $g'(n,h)$ given in Theorem 2.1(\romannumeral2) is in form of $g'(n,h)=f'(n+h)$ for some polynomial sequence $f':\mathbb{N}\rightarrow G'$.
}\end{remark}

\section{Proof of Theorem \ref{the theorem we should prove}}\label{the third main result}
Actually, Theorem \ref{the theorem we should prove} follows from a more general result (Theorem \ref{the estimate of general multiplicative functions} below). Before stating the result, let us recall the following distance function introduced by Granville and Soundararajan (see e.g., \cite{GS07}). Given integer $s\geq1$, for two multiplicative functions $f(n)$ and $g(n)$ with $|f(n)|,|g(n)|\leq 1$ for all $n\geq 1$, define
\[\mathbb{D}_{s}(f(n),g(n);X):=\Bigg(\sum_{\substack{p\leq X\\p \nmid s}}\frac{1-\text{Re}(f(p)\overline{g(p)})}{p}\Bigg)^{\frac{1}{2}}\]
This distance function was used in \cite{BGS} to measure the pretentiousness between any multiplicative function $f(n)$ and some function for which exceptional modulus $s$ does exist.
Throughout define
\begin{align}\label{definition of distance function}
\begin{aligned}
M_{s}(f;X,T):=&\inf_{\chi(mod~s)}\inf_{|t|\leq T}\mathbb{D}_{s}(f\chi,n\mapsto n^{it};X)^2,\\
M(f;X,T,Y):=&\inf_{s\leq Y}M_{s}(f;X,T)=\inf_{\substack{s\leq Y\\ \chi(mod~s)}}\inf_{|t|\leq T}\mathbb{D}_{s}(f\chi,n\mapsto n^{it};X)^2.
\end{aligned}
\end{align}
\begin{theorem}\label{the estimate of general multiplicative functions}
Let $s\geq 1$, $H\geq 3$ and $d\geq 0$ be integers. Let $\beta(n)$ be a multiplicative function with $|\beta(n)|\leq 1$ for all $n\geq 1$. Let $P(x)\in \mathbb{R}[x]$ have degree $d$ and $F:\mathbb{R}/\mathbb{Z}\rightarrow \mathbb{C}$ be a Lipschitz function. For $N$ large enough with $s\log H\leq \frac{1}{2}(\log N)^{1/32}$, we have
\begin{align}\label{1/18}
\begin{aligned}
&\sum_{a=1}^s \sum_{n=1}^{N}\left|\sum_{\substack{1\leq h\leq Hs\\n+h\equiv a(mod~s)}}\beta(n+h)F\big(P(n+h)(mod~\mathbb{Z})\big)\right|\ll NHs\frac{s}{\varphi(s)}\frac{\log\log H}{\log H}\\
&+ NHs\frac{1}{(\log N)^{1/200}}+(\log N)^{o(1)}NHs\frac{M\big(\beta;N/\log N,2N,(\log N)^{1/32}\big)+1}{\exp\Big(\frac{1}{2}M\big(\beta;N/\log N,2N,(\log N)^{1/32}\big)\Big)},
\end{aligned}
\end{align}
where the implied constant depends on $d$ and $\|F\|_{Lip}$ at most.
\end{theorem}
We remark that the constant implied in formula (\ref{1/18}) would be universal for $F$ from the family consisting of all Lipschitz functions whose Lipschitz norm is bounded by a given constant.

The factor $(\log N)^{o(1)}$ in (\ref{1/18}) comes from the restriction that the summation on  $n,h$ is over a certain dense set $\mathcal{S}$ of natural numbers (see Subsection \ref{the estimate of major arc case}). As pointed in the remark below \cite[Theorem 9.2]{KM2}, $(\log N)^{o(1)}M(\beta;\cdot,\cdot,\cdot)\exp(-\frac{1}{2}M(\beta;\cdot,\cdot,\cdot))$ would be $o_{N}(1)$ if $M(\beta;\cdot,\cdot,\cdot)$ tends to infinity faster than  $\log\log N$. In case $M(\beta;\cdot,\cdot,\cdot)$ tends to infinity slowly, it seems to be hard to show that $(\log N)^{o(1)}M(\beta;\cdot,\cdot,\cdot)\exp(-\frac{1}{2}M(\beta;\cdot,\cdot,\cdot))$ is small. For this case, we use a slightly different method to estimate the left-hand side of $(\ref{1/18})$ (see Remark \ref{another estimate about Theorem1}) to avoid the factor $(\log N)^{o(1)}$.

\begin{remark}\label{our method for non-abelian cases}
{\rm{Note that a simply connected abelian Lie group is isomorphic to $\mathbb{R}^{m}$, and any discrete subgroup of $\mathbb{R}^{m}$ is a lattice. So for the abelian nilmanifold, we may assume that $G=\mathbb{R}^{m}$ and $\Gamma=\mathbb{Z}^{m}$ after a linear transformation. It is not hard to check that by an argument similar to the proof of Theorem \ref{the estimate of general multiplicative functions}, we can extend $\mathbb{R}/\mathbb{Z}$ in Theorem \ref{the estimate of general multiplicative functions} to abelian nilmanifolds $G/\Gamma$, and extend $P(x)$ to polynomial sequences $p:\mathbb{Z}\rightarrow G$ in the form of $p(n)=a_{1}^{b_{1}(n)}\cdot\cdot\cdot a_{k}^{b_{k}(n)}$ for $a_{1},\ldots,a_{k}\in G$ and $b_{1}(x),\ldots,b_{k}(x)\in \mathbb{R}[x]$ with integer coefficients.}}
\end{remark}

Theorem \ref{the theorem we should prove} follows directly from Theorem \ref{the estimate of general multiplicative functions} through the following well-known result (see, e.g., \cite[(1.12)]{KMT17}) about the ``non-pretentious'' nature of the M\"{o}bius function.
\begin{proposition}\label{non-pretentious}
Let $N$ be large enough. Then for $\varepsilon>0$ sufficiently small,
\[
M(\mu;N/\log N,N,(\log N)^{1/32}) \geq (1/3-\epsilon) \log\log N+O(1).
\]
\end{proposition}
In the rest part of this section, we shall prove Theorem \ref{the estimate of general multiplicative functions}. We now outline the strategy in our proof. Using the quantitative factorization theorem (i.e., Theorem \ref{decomposition theorem}), we split the proof into two cases, the major and minor arc cases. This point is more clear if the Lipschitz function $F$ has zero mean. For this situation, the major arc part corresponds to $g'=0~(G'=\{0\})$, while the minor arc part corresponds to $g'\neq 0~(G'=\mathbb{R})$ in Theorem \ref{decomposition theorem}. In the major arc part, we use the result \cite[Theorem 1.6]{KMT} to sums of multiplicative functions over almost all short arithmetic progressions. In the minor arc case, we use a variant of the arguments of He-Wang \cite[Section 6]{HW} to obtain the cancellation we need.

Recall that we use $[N]$ to denote the set $\{1,\ldots,N\}$. In the following, for integers $n,a,s$, the notation $n\equiv a(s)$ means $n\equiv a$(mod~$s$).
In this section, many implicit constants $O(1)=O_{d}(1)\geq 1$ will appear and may not be the same in each occurrence. For convenience, we use a common constant $C_{0}=O_{d}(1)$ that is large enough for all these purposes. Let
\begin{equation}\label{parameter b}
B=100C_{0}^2, ~~~~~~~~~~~~~B_{2}=\frac{1}{10}C_{0}^{-2}B.
\end{equation}
The basic relation in this section among $s$, $H$, $N$ is as in Theorem \ref{the estimate of general multiplicative functions}, i.e.,
\[s\log H\leq \frac{1}{2}(\log N)^{1/32}.\]
For given $H$ sufficiently large, set
\begin{equation}\label{parameter choice}
W=\lfloor\log H\rfloor.
\end{equation}
Suppose $P(n)=\sum_{i=0}^{d}\alpha_{i}n^{i}$ with each $\alpha_{i}\in \mathbb{R}$ and $g(n,h)=P(n+h)$. Applying Theorem \ref{decomposition theorem} to $g(n,h)$ with $m=1$, there is a group $G'=\mathbb{R}$ or $\{0\}$, a set
\begin{equation}\label{defofdensesetn}
\mathcal{N}\subseteq [N]
\end{equation}
with $|\mathcal{N}|\geq (1-W^{-B/2})N$,
and a decomposition of
\begin{equation}\label{decomposition}
g(n,h) =\mathcal{E}(n,h) + g'(n,h) + \gamma(n,h)
\end{equation}
into polynomials $\mathcal{E}, g', \gamma: \mathbb{Z}^2\rightarrow \mathbb{R}$ with the following properties:

(\romannumeral1). $|\mathcal{E}(n,h+1) - \mathcal{E}(n,h)| \ll_d \frac{1}{sH/W}$ for $h\in [Hs]$.

(\romannumeral2). $g'(n,h)$ takes values in $G'$, and for any $n\in \mathcal{N}$, there are at least  $(1-W^{-B/2})s$ integers of $a\in[s]$ such that $\{g'(n,h)\Gamma'\}_{h\in \mathcal{A}_{a}}$ is totally $W^{-B}$-equidistributed in $G'/\Gamma'$, where $\mathcal{A}_{a}=\{1\leq h\leq Hs:h\equiv a(s)\}$ and $\Gamma'=G'\cap \mathbb{Z}$.

(\romannumeral3). For some integer $q$ with $1\leq q \leq W$, $\{\gamma(n,h)(mod~\mathbb{Z})\}_{(n,h)\in \mathbb{Z}^2}$ is $qs$-periodic.

In addition, we assume that $q\in (W/2,W]$, otherwise we consider a multiplier of $q$. Let $\beta:\mathbb{N}\rightarrow \mathbb{C}$ be a multiplicative function with $|\beta(n)|\leq 1$ for each $n\in \mathbb{N}$ and let $F:\mathbb{R}/\mathbb{Z}\rightarrow \mathbb{C}$ be a Lipschitz function with $\|F\|_{Lip}\leq 1$ (see equation (\ref{def of Lip}) for the definition of $\|\cdot\|_{Lip}$).
\subsection{Splitting into major and minor arc cases} Given $s\geq 1$, for any $a=1,\ldots,s$ and $n\in \mathbb{N}$, choose $\theta_{n,a}\in \mathbb{C}$ with $|\theta_{n,a}|=1$ such that
\begin{equation}\label{10/26formula1}
\Bigg| \sum_{\substack{h\leq Hs \\ n+h\equiv a(s)}}
\beta(n+h) F\big(g(n,h)(mod~\mathbb{Z})\big)\Bigg|=\theta_{n,a}\sum_{\substack{h\leq Hs \\ n+h\equiv a(s)}}
\beta(n+h) F\big(g(n,h)(mod~\mathbb{Z})\big).
\end{equation}
We first split the interval $[1,Hs]$ into disjoint subintervals $I_1,\ldots,I_{W^2}$, each has length $W^{-2}Hs$. To apply the periodicity of $\gamma(n,h)(mod~\mathbb{Z})$, we further split each subinterval $I_{t}$ into progressions $n+h\equiv js+a (qs)$ for $j\in \{0,\ldots,q-1\}$. For any given $j\in \{0,\ldots,q-1\}$ and $t\in \{1,\ldots,W^2\}$, define $\mathcal{E}_{a,n,j,t}=\mathcal{E}(n,h')$, where $h'$ is the smallest integer in $I_{t}$ with $n+h'\equiv js+a (qs)$. Then by property (\romannumeral1), for any $h\in I_{t}$,
\begin{equation}\label{10/26noon}
|\mathcal{E}_{a,n,j,t}-\mathcal{E}(n,h)|\leq \frac{W}{Hs}W^{-2}Hs=W^{-1}.
\end{equation}
Choose $\gamma_{a,n,j,t}\in [0,1)$ with $\gamma_{a,n,j,t}(mod~\mathbb{Z})=\gamma(n,h)(mod~\mathbb{Z})$ for any $h\in I_{t}$ satisfying $n+h\equiv js+a (qs)$. Note that $\gamma(n,h) (mod~\mathbb{Z})$ is $qs$-periodic, so $\gamma_{a,n,j,t}$ is independent of $h$ with  $n+h\equiv js+a (qs)$.

Define the function\footnote{In the general form, we should define $G_{a,n,j,t}=\gamma_{a,n,j,t}^{-1}G'\gamma_{a,n,j,t}$ and $\Gamma_{a,n,j,t}=G_{a,n,j,t}\cap \Gamma$. In the non-abelian case, for any given $j,t$, there may be $s^2$ distinct $G_{a,n,j,t}$ and $\Gamma_{a,n,j,t}$. Since we are in the abelian case, $G_{a,n,j,t}=G'$ and $\Gamma_{a,n,j,t}=G'\cap \Gamma=\Gamma'$.\label{footnote6}} $F_{a,n,j,t}:G'/\Gamma'\rightarrow \mathbb{C}$ as $F_{a,n,j,t}(y\Gamma')=\theta_{n,a}F\big((\mathcal{E}_{a,n,j,t}+\gamma_{a,n,j,t}+y)(mod~\mathbb{Z})\big)$ for any $y\in G'$. Note that $G'/\Gamma'=\mathbb{R}/\mathbb{Z}$ or $\{0\}$. It is easy to see that $F_{a,n,j,t}$ is well-defined. By formulas (\ref{decomposition}) and (\ref{10/26noon}),
\begin{equation}\label{noon 2}
|F_{a,n,j,t}(g'(n,h)\Gamma')-\theta_{n,a}F\big(g(n,h)(mod~\mathbb{Z})\big)|\leq W^{-1}
\end{equation}
holds for all $h\in I_{t}$ with $n+h\equiv js+a (qs)$.
Then by the above formula and formula (\ref{10/26formula1}),
\begin{align}\label{afternoon1}
\begin{aligned}
&\sum_{a=1}^s \sum_{n=1}^{N} \Bigg| \sum_{\substack{h\leq Hs \\ n+h\equiv a(s)}}
\beta(n+h) F\big(g(n,h)(mod~\mathbb{Z})\big)\Bigg| \\
=&\ds\sum_{a=1}^s \sum_{n=1}^{N}\sum_{t=1}^{W^2} \sum_{j=0}^{q-1} \sum_{\substack{h\in I_t \\ n+h\equiv js+a (qs)}}
\beta(n+h) F_{a,n,j,t}(g'(n,h)\Gamma') + O(W^{-1}HNs).
\end{aligned}
\end{align}
We write $F_{a,n,j,t}$ as $E_{a,n,j,t}+\widetilde{F}_{a,n,j,t}$, where $E_{a,n,j,t}=\int_{G'/\Gamma'}F_{a,n,j,t}$ is a constant and $\widetilde{F}_{a,n,j,t}$ has zero average on $G'/\Gamma'$. In the following, we shall apply different methods to estimate the major arc part
\begin{equation}\label{afternoon2}
\sum_{a=1}^s \sum_{n=1}^{N} \sum_{t=1}^{W^2} \sum_{j=0}^{q-1} \sum_{\substack{h\in I_t \\ n+h\equiv js+a (qs)}}
\beta(n+h)E_{a,n,j,t},
\end{equation}
and the minor arc part
\begin{equation}\label{afternoon3}
\sum_{a=1}^s \sum_{n=1}^{N} \sum_{t=1}^{W^2} \sum_{j=0}^{q-1} \sum_{\substack{h\in I_t \\ n+h\equiv js+a (qs)}}
\beta(n+h)\widetilde{F}_{a,n,j,t}(g'(n,h)\Gamma').
\end{equation}
\subsection{The estimate of major arc case.}\label{the estimate of major arc case}  For the major arc estimate, we shall show the following result.
\begin{theorem}\label{the estimate of major arc part}
Let $s\geq 1$ and $H\geq 3$ be integers. Let $N$ be large enough with $s\log H\leq \frac{1}{2}(\log N)^{1/32}$. Let $\beta(n)$ be a multiplicative function with $|\beta(n)|\leq 1$ for all $n\geq 1$. Then
\begin{align*}
  (\ref{afternoon2})\ll &NHs\Bigg(\frac{s}{\varphi(s)}\frac{\log\log H}{\log H}+\frac{1}{(\log N)^{1/200}}\Bigg) \\
  &+(\log N)^{o(1)}NHs\frac{M\big(\beta;N/\log N,2N,(\log N)^{1/32}\big)+1}{\exp\Big(\frac{1}{2}M\big(\beta;N/\log N,2N,(\log N)^{1/32}\big)\Big)}.
\end{align*}
\end{theorem}
We first bound formula (\ref{afternoon2}) by
\begin{align}\label{evening}
\begin{aligned}
&(\ref{afternoon2})\ll \sum_{a=1}^s \sum_{n=1}^{N}
\sum_{t=1}^{W^2} \sum_{j=0}^{q-1}
\left|\sum_{\substack{h\in I_t \\ n+h\equiv js+a(qs)}} \beta(n+h) \right|\\
&=\sum_{c=1}^{qs} \sum_{n=1}^{N}
\sum_{t=1}^{W^2}
\left|\sum_{\substack{h\in I_t \\ n+h\equiv c(qs)\\n+h\in \mathcal{S} }} \beta(n+h)\right|
+\sum_{c=1}^{qs} \sum_{n=1}^{N}
\sum_{t=1}^{W^2}
\left|\sum_{\substack{h\in I_t \\ n+h\equiv c(qs)\\n+h\notin \mathcal{S} }} \beta(n+h)\right|,
\end{aligned}
\end{align}
where $\mathcal{S}$ is given as follows. Let
\begin{equation}\label{choice of pandq}
P_{1}=(\log H)^{6000}, Q_{1}=H/(\log H)^4.
\end{equation}
Then we define $P_{j}$, $Q_{j}$ inductively in the following way, which is the same as \cite[Section 2]{KM},
\begin{equation}\label{1/17}
P_{j}=\exp(j^{4j}(\log Q_{1})^{j-1}\log P_{1}),~~Q_{j}=\exp(j^{4j+2}(\log Q_{1})^{j}).
\end{equation}
Let $N_{0}=\sqrt{N}$. Let $J$ be the largest index such that $Q_{J}\leq \exp\big(\sqrt{\log N_{0}}\big)$. For $j=1,\ldots,J$, denote by
\[\mathcal{P}_{j}=\{p:~\text{a~prime~in}~[P_{j},Q_{j}]~\text{with}~p\nmid s\}.\]
Define
\begin{equation}\label{definition of dense set1}
\mathcal{S}=\{n\leq 2N:n {\rm~has~at~least~one~prime~factor~in}~\mathcal{P}_{j} ~{\rm for~all~} j\in \{1,\ldots,J\}\}.
\end{equation}
We now estimate the contribution from $n+h\notin \mathcal{S}$ in formula (\ref{evening}).
\begin{lemma}\label{an estimate needed in the major arc}
Let $\mathcal{S}$ be as the above.
\[\sum_{c=1}^{qs} \sum_{n=1}^{N}
\sum_{t=1}^{W^2}
\left|\sum_{\substack{h\in I_t \\ n+h\equiv c(qs)\\n+h\notin \mathcal{S} }} \beta(n+h)\right|\ll NHs\frac{s}{\varphi(s)}\frac{\log\log H}{\log H}. \]
\end{lemma}
\begin{proof}
By the fundamental sieve, for $j=1,\ldots,J$, we have
\begin{equation}\label{fundamental sieve}
\sum_{\substack{n\leq X\\ \big(n,~\prod_{{P_{j}< p\leq Q_{j}, p\nmid s}}~p\big)=1}}1\ll X\frac{\log P_{j}}{\log Q_{j}}\prod_{\substack{P_{j}<p\leq Q_{j}\\p|s }}(1-1/p)^{-1}+O(sQ_{j}^2).
\end{equation}
So by the above and choices (\ref{choice of pandq}), (\ref{1/17}),
\begin{align*}
&\sum_{c=1}^{qs} \sum_{n=1}^{N}
\sum_{t=1}^{W^2}
\left|\sum_{\substack{h\in I_t \\ n+h\equiv c(qs)\\n+h\notin \mathcal{S} }} \beta(n+h)\right|\leq \sum_{n=1}^{N}\sum_{\substack{m=n\\m\notin \mathcal{S}}}^{n+Hs}1=\sum_{\substack{m=1\\m\notin \mathcal{S}}}^{N+Hs}\sum_{n=m-Hs}^{m}1\\
&\leq Hs\sum_{j=1}^{J}\sum_{\substack{n\leq N+Hs\\ (n,~\prod_{{P_{j}< p\leq Q_{j},p\nmid s}}~p)=1}}1
\ll NHs\frac{\log P_{1}}{\log Q_{1}}\frac{s}{\varphi(s)}\\
&\ll NHs\frac{s}{\varphi(s)}\frac{\log\log H}{\log H}
\end{align*}
as claimed.
\end{proof}
Hence, to prove Theorem \ref{the estimate of major arc part}, we are left to bound
\begin{equation}\label{11/8evening}
\sum_{c=1}^{qs} \sum_{n=1}^{N}
\sum_{t=1}^{W^2}
\left|\sum_{\substack{h\in I_t \\ n+h\equiv c(qs)\\n+h\in \mathcal{S} }} \beta(n+h)\right|.
\end{equation}
Let $\beta_{1}$ denote the completely multiplicative function defined by $\beta_{1}(p) = \beta(p)$ for all prime numbers $p$. We first bound the above in terms of $\beta_{1}$ rather than $\beta$ by a standard technique (see, e.g., \cite[p2177-p2178]{KMT17}). The Dirichlet inverse of $\beta_{1}$ is $\mu  \beta_{1}$. Then we can write $\beta = \beta_{1}*\alpha$, where $*$ is the Dirichlet convolution and $\alpha = \beta * \mu \beta_{1}$. Observe that $\alpha(n)$ is multiplicative and for all primes $p$, $\alpha(p)=0$ and $|\alpha(p^{j})|\leq 2$ for $j\geq 2$. So from the Euler product, we see that $\sum_{n=1}^{\infty} |\alpha(n)| n^{-\frac{3}{4}} =O(1)$. Note that $W<P_{1}$. So $1_{\mathcal{S}}(ab)=1_{\mathcal{S}}(b)$ when $a\leq W$ and $b\leq (N+Hs)/a$.
Then
\begin{align}
(\ref{11/8evening})=&\sum_{c=1}^{qs} \sum_{n=1}^{N}
\sum_{t=1}^{W^2} \sum_{a\leq W}|\alpha(a)|
\left|\sum_{\substack{b\in \mathbb{N}\\ab\equiv c(qs) \\ ab\in n+I_t}}1_{\mathcal{S}}(b)\beta_{1}(b) \right|+
\sum_{c=1}^{qs} \sum_{n=1}^{N}
\sum_{t=1}^{W^2} \sum_{a> W}|\alpha(a)|
\left|\sum_{\substack{b\in \mathbb{N}\\ ab\equiv c(qs) \\ ab\in n+I_t}}1_{\mathcal{S}}(ab) \beta_{1}(b) \right| \nonumber\\
=&\sum_{c=1}^{qs} \sum_{n=1}^{N}
\sum_{t=1}^{W^2} \sum_{a\leq W}
|\alpha(a)|
\left|\sum_{\substack{b\in \mathbb{N}, ab\equiv c(qs) \\ ab\in n+I_t}}1_{\mathcal{S}}(b)\beta_{1}(b) \right|
+O\left(
\sum_{n=1}^{N}
\sum_{t=1}^{W^2} \sum_{a> W}
|\alpha(a)|
\Bigg(\sum_{b\in \mathbb{N}, ab\in n+I_t} 1 \Bigg) \right) \nonumber\\
=&\sum_{c=1}^{qs} \sum_{n=1}^{N}
\sum_{t=1}^{W^2} \sum_{a\leq W}
|\alpha(a)|
\left|\sum_{\substack{b\in \mathbb{N},ab\equiv c(qs) \\ ab\in n+I_t}}1_{\mathcal{S}}(b)\beta_{1}(b) \right|
+O\Big( W^{-1/4} NHs \Big).\label{evening2}
\end{align}
Now the estimate of formula (\ref{afternoon2}) is reduced to estimating
\begin{equation}\label{10/29noonformula1}
\sum_{c=1}^{qs} \sum_{n=1}^{N}
\sum_{t=1}^{W^2} \sum_{a\leq W}
|\alpha(a)|
\left|\sum_{\substack{b\in \mathbb{N},ab\equiv c(qs)\\ ab\in n+I_t}}1_{\mathcal{S}}(b)\beta_{1}(b) \right|.
\end{equation}
We split the summation over $b$ according to gcd$(a,qs)$. Note that $|\alpha(n)|\leq \tau(n)$ (the divisor function) for all $n\in \mathbb{N}$.
Then
\begin{align}\label{10/29noonformula2}
\begin{aligned}
(\ref{10/29noonformula1})=&\sum_{t\leq W^2} \sum_{\substack{e\leq W\\e\mid qs}} \sum_{\substack{a\leq W/e \\ (a,qs/e) = 1}}|\alpha(ae)|\sum_{{1\leq c \leq qs/e}} \sum_{n\leq N}  \left|\sum_{\substack{b\in n/ae + I_t/ae \\ b\equiv c(qs/e) }}1_{\mathcal{S}}(b)\beta_{1}(b) \right|\\
=&\sum_{t\leq W^2} \sum_{\substack{e\leq W\\ e\mid qs}} \sum_{\substack{a\leq W/e \\ (a,qs/e) = 1}} |\alpha(ae)|ae \sum_{{1\leq c \leq qs/e}} \sum_{n\leq N/ae}  \left|\sum_{\substack{b\in n + I_t/ae \\ b\equiv c(qs/e) }}1_{\mathcal{S}}(b)\beta_{1}(b) \right|+O\Big(
NsW^3q(\log W)^2\Big).
\end{aligned}
\end{align}
To give an upper bound for (\ref{10/29noonformula2}), we apply the following estimate on sums of multiplicative functions in short arithmetic progressions.
\begin{lemma}\label{a lemma need to be used}
Given $X$ with $\sqrt{N}\leq X\leq N/2$ and $1\leq s\leq (\log X)^{1/32}$. Let $3\leq H\leq X/s$. Let $f(n)$ be a multiplicative function with $|f(n)|\leq 1$ for all $n\geq 1$. Then
\begin{align}\label{average on multiplicative function}
\sum_{\substack{a=1\\(a,s)=1}}^{s}\sum_{x=X}^{2X}\Bigg|\sum_{\substack{n=x\\n\equiv a(mod~s)}}^{x+Hs}1_{\mathcal{S}}(n)f(n)\Bigg|^2 \ll&H^2X\varphi(s)\Bigg(\frac{(\log Q_{1})^{1/3}}{P_{1}^{1/6-\eta}}+2^{2J}\frac{M(f;X,2X,s)+1}{\exp\Big(M(f;X,2X,s)\Big)}\Bigg)\nonumber \\
&+H^2X\varphi(s)\frac{1}{(\log X)^{1/65}},
\end{align}
where $\eta=1/150$, $P_{1},Q_{1}$ and $\mathcal{S}$, $J$ are given in (\ref{choice of pandq}) and (\ref{definition of dense set1}), respectively.
\end{lemma}
\begin{proof}
The proof is almost identical to the proof of \cite[Theorem 1.6]{KMT} (see also \cite[Proposition 6.1]{Wei2}), in which it gives an upper bound of left-hand side of (\ref{average on multiplicative function}) without restricting $n$ to $\mathcal{S}$. The difference is that when we apply the Hal\'{a}sz-type inequality to a Dirichlet polynomial of the form $\sum_{X\leq n\leq 2 X}\frac{1_{\mathcal{S}}(n)g(n)\chi(n)}{n^{1+it}}$ for some function $g(n)$, we should rewrite
\[\sum_{X\leq n\leq 2X}\frac{1_{\mathcal{S}}(n)g(n)\chi(n)}{n^{1+it}}=\sum_{\mathcal{J}\subseteq \{1,\ldots,J\}}(-1)^{|\mathcal{J}|}\sum_{X\leq n\leq 2X}\frac{1_{C_{\mathcal{J}}}(n)g(n)\chi(n)}{n^{1+it}},\]
where $C_{\mathcal{J}}$ is the set $\{n\in \mathbb{N}:(n,\prod_{p\in \cup_{j\in \mathcal{J}}\mathcal{P}_{j}}p)=1\}$ and $1_{C_{\mathcal{J}}}(n)$ is a multiplicative function (see \cite[Lemma 5]{KM}).
\end{proof}
By the above result, we have the following result.
\begin{corollary}\label{a corollary need to be used}
Let $s\geq 1$. Let $N_{1}$ be large enough with $N^{2/3}\leq N_{1}\leq N$ and $s\leq \frac{3}{4}(\log N_{1})^{1/32}$. Let $\beta_{1}(n)$ be a completely multiplicative function with $|\beta_{1}(n)|\leq 1$ for all $n\geq 1$. Then for any $H$ with $3\leq H\leq \frac{N_{1}}{s(\log N_{1})^{1/3}}$,
\begin{align}\label{average on multiplicative function2}
\begin{aligned}
 \sum_{a=1}^{s}\sum_{x=1}^{N_{1}}\Bigg|\sum_{\substack{n=x\\n\equiv a(mod~s)}}^{x+Hs}1_{\mathcal{S}}(n)\beta_{1}(n)\Bigg|\ll& N_{1}Hs\Big(\frac{(\log Q_{1})^{1/6}}{P_{1}^{1/12-\eta/2}}+\frac{1}{(\log N_{1})^{1/200}} \Big)\\
  &+N_{1}Hs2^{J}\Bigg(\frac{M(\beta_{1};N_{1}/(\log N_{1})^{2/3},N_{1},s)+1}{\exp\Big(M(\beta_{1};N_{1}/(\log N_{1})^{2/3},N_{1},s)\Big)}\Bigg)^{1/2}.
\end{aligned}
\end{align}
\end{corollary}
\begin{proof}
By the Cauchy-Schwarz inequality,
\begin{equation}\label{10/29afternoon9}
\sum_{a=1}^{s}\sum_{x=N_{1}/(\log N_{1})^{1/3}}^{N_{1}}\Bigg|\sum_{\substack{n=x\\n\equiv a(mod~s)}}^{x+Hs}1_{\mathcal{S}}(n)\beta_{1}(n)\Bigg|\leq (sN_{1})^{\frac{1}{2}}\Bigg(\sum_{a=1}^{s}\sum_{x=N_{1}/(\log N_{1})^{1/3}}^{N_{1}}\Bigg|\sum_{\substack{n=x\\n\equiv a(mod~s)}}^{x+Hs}1_{\mathcal{S}}(n)\beta_{1}(n)\Bigg|^2\Bigg)^{\frac{1}{2}}.
\end{equation}
Note that $M(f;X,T,Y)$ defined as in equation (\ref{definition of distance function}) is increasing in $X$ and decreasing both in $T$ and in $Y$, and $(x+1)e^{-x}$ is a decreasing function of $x$ when $x\in [0,\infty)$. Then by Lemma \ref{a lemma need to be used},
\begin{align*}
&\sum_{a=1}^{s}\sum_{x=N_{1}/(\log N_{1})^{1/3}}^{N_{1}}\Bigg|\sum_{\substack{n=x\\n\equiv a(mod~s)}}^{x+Hs}1_{\mathcal{S}}(n)\beta_{1}(n)\Bigg|^2\\
&=\sum_{d|s}\sum_{\substack{a=1\\(a,s/d)=1}}^{s/d}\sum_{x=N_{1}/(\log N_{1})^{1/3}}^{N_{1}}\Bigg|\sum_{\substack{n=x/d\\n\equiv a(mod~s/d)}}^{x/d+Hs/d}1_{\mathcal{S}}(dn)\beta_{1}(dn)\Bigg|^2\\
&\leq\sum_{d|s}d\sum_{\substack{a=1\\(a,s/d)=1}}^{s/d}\sum_{x=N_{1}/(d(\log N_{1})^{1/3})}^{N_{1}/d}\Bigg|\sum_{\substack{n=x\\n\equiv a(mod~s/d)}}^{x+Hs/d}1_{\mathcal{S}}(n)\beta_{1}(n)\Bigg|^2+O(N_{1}s)\\
&\ll H^2N_{1}\Big(\sum_{d|s}\varphi(\frac{s}{d})\Big)\Bigg(\frac{(\log Q_{1})^{1/3}}{P_{1}^{1/6-\eta}}+\frac{1}{(\log N_{1})^{1/65}}+2^{2J}\frac{M(\beta_{1};N_{1}/(\log N_{1})^{2/3},N_{1},s)+1}{\exp\Big(M(\beta_{1};N_{1}/(\log N_{1})^{2/3},N_{1},s)\Big)}\Bigg)\\
&\ll  H^2N_{1}s\Bigg(\frac{(\log Q_{1})^{1/3}}{P_{1}^{1/6-\eta}}+\frac{1}{(\log N_{1})^{1/65}}+2^{2J}\frac{M(\beta_{1};N_{1}/(\log N_{1})^{2/3},N_{1},s)+1}{\exp\Big(M(\beta_{1};N_{1}/(\log N_{1})^{2/3},N_{1},s)\Big)}\Bigg).
\end{align*}
In the first inequality above, we used the facts that $\beta_{1}(dn)=\beta_{1}(d)\beta_{1}(n)$ by the complete multiplicity of $\beta_{1}$ and $1_{\mathcal{S}}(dn)=1_{\mathcal{S}}(n)$ when $n\leq \frac{2N}{d}$ and $d|s$ by the definition of $\mathcal{S}$. So by formula (\ref{10/29afternoon9}), we obtain formula (\ref{average on multiplicative function2}).
\end{proof}
According to the previous analysis and Corollary \ref{a corollary need to be used}, we now prove Theorem \ref{the estimate of major arc part}.
\begin{proof}[Proof of Theorem \ref{the estimate of major arc part}]
Note that $q\leq W$ and $W=\lfloor \log H\rfloor$ (i.e., formula (\ref{parameter choice})). So for $N,H$ sufficiently large, $s\log H\leq \frac{1}{2}(\log N)^{1/32}$ implies that $qs\leq \frac{3}{4}\big(\log (N/W)\big)^{1/32}$. Observe that $\frac{(\log Q_{1})^{1/6}}{P_{1}^{1/12-1/300}}$ in (\ref{average on multiplicative function2}) is $O\Big(\frac{1}{(\log H)^{400}}\Big)$ by the choices (\ref{choice of pandq}). Applying Corollary \ref{a corollary need to be used} to formula (\ref{10/29noonformula2}), we have that for $N$ large enough with $s\log H\leq \frac{1}{2}(\log N)^{1/32}$,
\begin{align*}
&\sum_{t\leq W^2} \sum_{\substack{e\leq W\\ e\mid qs}} \sum_{\substack{a\leq W/e \\ (a,qs/e) = 1}}|\alpha(ae)| ae \sum_{{1\leq c \leq qs/e}} \sum_{n\leq N/ae}  \Bigg|\sum_{\substack{b\in n + I_t/ae \\ b\equiv c(qs/e) }}1_{\mathcal{S}}(b)\beta_{1}(b) \Bigg|\\
\ll&\sum_{t\leq W^2} \sum_{\substack{e\leq W\\ e\mid qs}} \sum_{\substack{a\leq W/e \\ (a,qs/e) = 1}} |\alpha(ae)|ae(\frac{qs}{e}\frac{N}{ae}\frac{W^{-2}Hs}{aqs})\Bigg(\frac{1}{(\log H)^{400}}+\frac{1}{(\log N)^{1/200}}\\
&2^{J}\Big(M(\beta_{1};N/\log N,2N,qs)+1\Big)^{1/2}\exp\Big(-\frac{1}{2}M(\beta_{1};N/\log N,2N,qs)\Big)\Bigg)\\
\ll&NHs\Bigg(\sum_{m\leq W}\frac{|\alpha(m)|\tau(m)}{m}\Bigg)\Bigg(\frac{1}{(\log H)^{400}}+\frac{1}{(\log N)^{1/200}}\\
&+2^{J}\frac{M(\beta_{1};N/\log N,2N,(\log N)^{1/32})+1}{\exp\Big(\frac{1}{2}M(\beta_{1};N/\log N,2N,(\log N)^{1/32})\Big)}\Bigg)\\
\ll&NHs\Bigg(\frac{1}{(\log H)^{400}}+\frac{1}{(\log N)^{1/200}}+
(\log N)^{o(1)}\frac{M(\beta_{1};N/\log N,2N,(\log N)^{1/32})+1}{\exp\Big(\frac{1}{2}M(\beta_{1};N/\log N,2N,(\log N)^{1/32})\Big)}\Bigg).
\end{align*}
In the last inequality we used the estimate $2^{J}\ll (\log N)^{o(1)}$ by the definition of $J$ after equation (\ref{1/17}), and $\sum_{m=1}^{\infty}\frac{|\alpha(m)|\tau(m)}{m}=O(1)$ by the Euler product and the fact that $M(\beta;X,T,Y)=M(\beta_{1};X,T,Y)$ by the definition of $\beta_{1}$ after formula (\ref{11/8evening}). By the above and Lemma \ref{an estimate needed in the major arc}, together with formulas (\ref{evening}), (\ref{evening2}), (\ref{10/29noonformula2}), we obtain the conclusion stated in this theorem.
\end{proof}
\subsection{The estimate of minor arc case.} For the minor arc estimate, we shall prove the following result.
\begin{theorem}\label{minor arc estimate}
Let $s\geq 1$, $H\geq 3$ and $N\geq Hs$ be integers. Let $\beta:\mathbb{N}\rightarrow \mathbb{C}$ be a multiplicative function with $|\beta(n)|\leq 1$ for each $n\in \mathbb{N}$. Then we have
\[(\ref{afternoon3})=\sum_{a=1}^s \sum_{n=1}^{N} \sum_{t=1}^{W^2} \sum_{j=0}^{q-1}\sum_{\substack{h\in I_t \\ n+h\equiv js+a (qs)}}
\beta(n+h)\widetilde{F}_{a,n,j,t}(g'(n,h)\Gamma')\ll NHs\frac{s}{\varphi(s)}\frac{\log\log H}{\log H}.\]
\end{theorem}
Let $P,Q$ be parameters to be determined later with $W<P<Q<H$. Given an integer $s\geq 1$, define $\mathcal{P}$ as the set of primes $p$ in $(P,Q]$ which are coprime to $s$. Note that here in our setting $(p,s)=1$ in $\mathcal{P}$, which is a little different from that of \cite[Section 6]{HW}. Let  $\widetilde{\mathcal{S}}=\{n: 1\leq n\leq N, \exists p\in \mathcal{P}, p\mid n\}$. Similar to the proof of \cite[Lemma 6.2]{HW}, we have
\begin{equation}\label{afternoon4}
\sum_{\substack{ h\leq Hs \\ n+h\in \widetilde{\mathcal{S}}  \\ n+h \equiv a(s)}} \Bigg| \beta(n+h) - \sum_{p\in \mathcal{P}}\sum_{l\in \mathbb{N}}\frac{1_{pl=n+h} \beta(p) \beta(l)}{1+|\{q\in\mathcal{P}: q\mid l\}|}\Bigg|  \ll \frac{H}{P}.
\end{equation}
By the above inequality, we conclude the following result.
\begin{lemma}\label{evening formula1anfafternon5}
The expression
\begin{equation}\label{10/28/evening formula1}
\sum_{a=1}^s \sum_{n=1}^{N} \sum_{t=1}^{W^2} \sum_{j=0}^{q-1}\sum_{\substack{h\in I_t \\ n+h\equiv js+a (qs)}}1_{\widetilde{\mathcal{S}}}(n+h)
\beta(n+h)\widetilde{F}_{a,n,j,t}(g'(n,h)\Gamma')
\end{equation}
is approximated by
\begin{equation}\label{afternoon5}
\sum_{a=1}^s \sum_{n\in \mathcal{N}}\sum_{t=1}^{W^2} \sum_{j=0}^{q-1} \sum_{p\in \mathcal{P}} \sum_{\substack{l\in \mathbb{N} \\ pl\equiv js+a (qs)}}
\frac{1_{pl\in n+I_t} \beta(p) \beta(l)}{1+|\{q\in\mathcal{P}: q\mid l\}|} \widetilde{F}_{a,n,j,t}(g'(n,pl-n)\Gamma')
\end{equation}
up to an error $O(P^{-1}HNs+W^{-B/2}HNs)$, where $\mathcal{N}$ is the set defined in equation (\ref{defofdensesetn}).
\end{lemma}
In the following, we show that the estimate of (\ref{afternoon5}) is reduced to proving (\ref{10/28formula4}).
Let
\begin{equation}\label{definition of P and Q}
P=2^{r_{-}},~~Q=2^{r_{+}},~~r_{-}=\lfloor 2\log\log H\rfloor+1,~~r_{+}=\lfloor 1/10 \log H\rfloor+1.
\end{equation}
Then (\ref{afternoon5}) becomes
\begin{equation}\label{afternoon6}
\sum_{ r\in(r_{-} , r_{+}]}
\sum_{a=1}^s \sum_{n\in \mathcal{N}} \sum_{t=1}^{W^2} \sum_{j=0}^{q-1}\sum_{\substack{p\in (2^{r-1},2^{r}]\\(p,s)=1}}   \sum_{\substack{l\in \mathbb{N} \\ pl\equiv js+a (qs)}}
\frac{1_{pl\in n+I_t} \beta(p) \beta(l)}{1+|\{q\in\mathcal{P}: q\mid l\}|} \widetilde{F}_{a,n,j,t}(g'(n,pl-n)\Gamma').
\end{equation}
Given an integer $r\in (r_{-},r_{+}]$ and a prime $p\in (2^{r-1},2^{r}]$ with $(p,s)=1$, by the choice of $P$ and $W$, we see that $(p,q)=1$. Fix an integer $\overline{p}$ with $\overline{p}p\equiv 1(qs)$. Given $a=1,\ldots,s$ and $j=0,\ldots,q-1$, let $b_{a}\in [1,s]$ with $\overline{p}b_{a}\equiv a(s)$. Write $\overline{p}b_{a}=l_{a}s+a$ for some integer $l_{a}$. Define $J_{p,a,j}$ as the integer in $[0,q-1]$ with
\begin{equation}\label{definition of capitalj}
J_{p,a,j}\equiv p(j-l_{a})(mod~q).
\end{equation}
In the following, for simplicity, we use $\textbf{j}$ to denote $J_{p,a,j}$. Now the expression inside the summation over $r$ in formula (\ref{afternoon6}) becomes
\begin{align}\label{10/28formula1}
\begin{aligned}
&\sum_{\substack{p\in (2^{r-1},2^{r}]\\(p,s)=1}}\sum_{n\in \mathcal{N}} \sum_{t=1}^{W^2}\sum_{a=1}^s \sum_{j=0}^{q-1}\sum_{\substack{l\in \mathbb{N} \\ l\equiv \overline{p}(js+a) (qs)}}
\frac{1_{pl\in n+I_t} \beta(p) \beta(l)}{1+|\{q\in\mathcal{P}: q\mid l\}|} \widetilde{F}_{a,n,j,t}(g'(n,pl-n)\Gamma')\\
=&\sum_{\substack{p\in (2^{r-1},2^{r}]\\(p,s)=1}}\sum_{n\in \mathcal{N}} \sum_{t=1}^{W^2} \sum_{a=1}^s \sum_{j=0}^{q-1}\sum_{\substack{l\in \mathbb{N} \\ l\equiv js+a (qs)}}
\frac{1_{pl\in n+I_t} \beta(p) \beta(l)}{1+|\{q\in\mathcal{P}: q\mid l\}|} \widetilde{F}_{pa,n,\textbf{j},t}(g'(n,pl-n)\Gamma').
\end{aligned}
\end{align}
Here we view the index $pa$ appearing in $\widetilde{F}$ as $pa(mod~s)$.
By Cauchy-Schwarz inequality, we can bound equation (\ref{10/28formula1}) by
\begin{align}
&\leq \sum_{a=1}^s \sum_{j=0}^{q-1}\sum_{t=1}^{W^2} \sum_{\substack{l\leq \frac{N+Hs}{2^{r-1}} \\ l\equiv js+a (qs)}}   \frac{ |\beta(l)|}
{1+|\{q\in\mathcal{P}: q\mid l\}|}\Bigg|\sum_{n\in \mathcal{N}} \sum_{\substack{p\in (2^{r-1},2^{r}]\\pl\in n+I_{t}\\(p,s)=1}}
\beta(p)\widetilde{F}_{pa,n,\textbf{j},t}(g'(n,pl-n)\Gamma')\Bigg| \nonumber\\
&\leq \left(\frac{N}{2^{r-1}}W^2\right)^{1/2}
\Bigg(\sum_{a=1}^s \sum_{j=0}^{q-1}\sum_{t=1}^{W^2} \sum_{\substack{ l\leq \frac{N+Hs}{2^{r-1}} \\ l\equiv js+a(qs)}} \Bigg|
\sum_{n\in \mathcal{N}}\sum_{\substack{ p\in(2^{r-1},2^r] \\ pl \in n + I_t \\ (p,s)= 1}}
\beta(p)\widetilde{F}_{pa,n,\textbf{j},t}(g'(n,pl-n)\Gamma')\Bigg|^2 \Bigg)^{1/2} \nonumber \\
&=\left(\frac{N}{2^{r-1}}W^2\right)^{1/2}\cdot S^{1/2},\label{10/28formula3}
\end{align}
where \[S:=\sum_{a=1}^s\sum_{j=0}^{q-1}\sum_{t=1}^{W^2} \sum_{\substack{ l\leq \frac{N+Hs}{2^{r-1}} \\ l\equiv js+a(qs)}} \Bigg|
\sum_{n\in \mathcal{N}} \sum_{\substack{ p\in(2^{r-1},2^r] \\ pl \in n + I_t \\ (p,s)= 1}}
\beta(p)\widetilde{F}_{pa,n,\textbf{j},t}(g'(n,pl-n)\Gamma')\Bigg|^2.\]
Now the proof of Theorem \ref{minor arc estimate} is reduced to proving
\begin{equation}\label{10/28formula4}
S\ll2^r W^{-(B_2+2)} H^2 N s^2,
\end{equation}
where the parameter $B_{2}$ is given in formula (\ref{parameter b}).
\begin{proof}[Proof of Theorem \ref{minor arc estimate}](Assume formula (\ref{10/28formula4}))
By formulas (\ref{afternoon6}), (\ref{10/28formula1}), (\ref{10/28formula3}) and (\ref{10/28formula4}), we have
\[\sum_{a=1}^s \sum_{n\in \mathcal{N}}\sum_{t=1}^{W^2} \sum_{j=0}^{q-1} \sum_{p\in \mathcal{P}} \sum_{\substack{l\in \mathbb{N} \\ pl\equiv js+a (qs)}}
\frac{1_{pl\in n+I_t} \beta(p) \beta(l)}{1+|\{q\in\mathcal{P}: q\mid l\}|} \widetilde{F}_{a,n,j,t}(g'(n,pl-n)\Gamma')\ll r_{+}W^{-\frac{B_{2}}{2}}HNs.\]
Then by Lemma \ref{evening formula1anfafternon5},
\[\sum_{a=1}^s \sum_{n=1}^{N} \sum_{t=1}^{W^2} \sum_{j=0}^{q-1}\sum_{\substack{h\in I_t \\ n+h\equiv js+a (qs)}}1_{\widetilde{\mathcal{S}}}(n+h)
\beta(n+h)\widetilde{F}_{a,n,j,t}(g'(n,h)\Gamma')\ll (r_{+}W^{-\frac{B_{2}}{2}}+P^{-1}+W^{-\frac{B}{2}})HNs.\]
We know that $B=100C_{0}^2$, $B_{2}=\frac{1}{10}C_{0}^{-2}B$ by formula (\ref{parameter b}), $P=2^{\lfloor 2\log\log H\rfloor+1}$, $Q=2^{\lfloor 1/10 \log H\rfloor+1}$ by (\ref{definition of P and Q}), and $W=\lfloor \log H \rfloor$ by (\ref{parameter choice}).
Note that $B\geq 10$. Then
\begin{equation}\label{evening3}
\sum_{a=1}^s \sum_{n=1}^{N} \sum_{t=1}^{W^2} \sum_{j=0}^{q-1}\sum_{\substack{h\in I_t \\ n+h\equiv js+a (qs)}}1_{\widetilde{\mathcal{S}}}(n+h)
\beta(n+h)\widetilde{F}_{a,n,j,t}(g'(n,h)\Gamma')\ll (\log H)^{-1}NHs.
\end{equation}
We are left with estimating the contribution from
\begin{equation}\label{evening4}
\sum_{a=1}^s \sum_{n=1}^{N} \sum_{t=1}^{W^2} \sum_{j=0}^{q-1}\sum_{\substack{h\in I_t \\ n+h\equiv js+a (qs)\\n+h\notin \widetilde{\mathcal{S}}}}
\beta(n+h)\widetilde{F}_{a,n,j,t}(g'(n,h)\Gamma').
\end{equation}
By the fundamental sieve (see formula (\ref{fundamental sieve})), we have
\[(\ref{evening4})\leq \sum_{n=1}^{N}\sum_{\substack{m=n\\m\notin \widetilde{\mathcal{S}}}}^{n+Hs}1=\sum_{\substack{m=1\\m\notin \widetilde{\mathcal{S}}}}^{N+Hs}\sum_{n=m-Hs}^{m}1\ll NHs\frac{\log P}{\log Q}\frac{s}{\varphi(s)}\ll NHs\frac{s}{\varphi(s)}\frac{\log\log H}{\log H}.\]
\end{proof}
In the end, we focus on showing formula (\ref{10/28formula4}).  Denote by $\mathcal{A}_{n,p,t} =\{l\in \mathbb{N}: pl-n\in I_t\}$ and $\mathcal{A}_{n_{1},n_{2},p_{1},p_{2},t}=\mathcal{A}_{n_{1},p_{1},t}\cap \mathcal{A}_{n_{2},p_{2},t}$. It is not hard to check that $|\mathcal{A}_{n_{1},n_{2},p_{1},p_{2},t}|\ll 2^{-r}W^{-2}Hs$. For given $t\in[W^2]$, $n_{1}\in \mathcal{N},p_{1},p_{2}\in (2^{r-1},2^{r}]$,
\begin{equation}\label{11/7morning}
|\{n_{2}:n_{2}\in \mathcal{N}:\mathcal{A}_{n_{1},n_{2},p_{1},p_{2},t}\neq \emptyset\}|\leq W^{-2}Hs.
\end{equation}
Expanding the square term in $S$, we have that $S$ is bounded by
\begin{align}\label{10/28formula5}
\begin{aligned}
&\sum_{a=1}^s\sum_{n_{1}, n_{2}\in \mathcal{N}} \sum_{ \substack{ p_1,p_2 \in(2^{r-1}, 2^r] \\ (p_1p_2, s) =1}}
 \sum_{j=0}^{q-1}\sum_{t=1}^{W^2}\\
&\Bigg|\sum_{\substack{l\in \mathcal{A}_{n_{1},n_{2},p_{1},p_2,t} \\ l\equiv js+a(qs)}}\widetilde{F}_{p_{1}a,n_{1},\textbf{j}_{1},t}(g'(n_{1},p_{1}l-n_1)\Gamma')
\overline{\widetilde{F}_{p_{2}a,n_{2},\textbf{j}_{2},t}}(g'(n_{2},p_{2}l-n_2)\Gamma') \Bigg|,
\end{aligned}
\end{align}
where $\textbf{j}_{1}=J_{p_{1},a,j}$ and $\textbf{j}_{2}=J_{p_{2},a,j}$ defined as equation $(\ref{definition of capitalj})$.
To prove formula (\ref{10/28formula4}), it is enough to prove
\begin{equation}\label{11/4formula1}
(\ref{10/28formula5})\ll2^r W^{-(B_2+2)} H^2 N s^2.
\end{equation}
From Appendix \ref{aproofofaproposition}, the proof of formula (\ref{11/4formula1}) can be reduced to proving the following proposition, which is a variant version of \cite[Proposition 6.9]{HW} in arithmetic progressions that is adapted to our situation.
\begin{proposition}\label{small portion with not equidistributed}
Let $C_{0}$ be a large enough constant only depending on $d$, and let $B_{2}\geq 10$, $B\geq 10C_{0}^2B_{2}$, $H\geq \max(W^{B},2^{10r_{+}})$. Given $s\geq 1$, $j=0,\ldots,q-1$ and $t=1,\ldots,W^2$. Let $r_{-}$, $r_{+}$ be as in formula (\ref{definition of P and Q}). For $r\in (r_{-},r_{+}]$, $n_{1}\in \mathcal{N}$ and a prime number $p_{1}\in (2^{r-1},2^{r}]$ with $(p_{1},s)=1$,
the set $\Omega_{r,n_1,p_1,B_2}$ is defined to be the set of all triples:
\[
(a,n,p) \in [1,s] \times \mathcal{N}\times \{p:p\in (2^{r-1},2^{r}],(p,s)=1\}.
\]
such that

(\romannumeral1) $p$ is prime;

(\romannumeral2) $|\mathcal{A}_{n_1,n,p_1,p,t} |\geq 2^{-r} W^{-(B_2+2)} Hs $;

(\romannumeral3)
\[\Bigg|\sum_{\substack{l\in \mathcal{A}_{n_{1},n,p_{1},p,t} \\ l\equiv js+a(qs)}}\widetilde{F}_{p_{1}a,n_{1},\textbf{j}_{1},t}(g'(n_{1},p_{1}l-n_1)\Gamma')
\overline{\widetilde{F}_{pa,n,\textbf{j},t}}(g'(n,pl-n)\Gamma') \Bigg|\geq \frac{W^{-B_2}|\mathcal{A}_{n_{1},n,p_{1},p,t}|}{qs}.\]
Then
\begin{equation}\label{thecardinalityofdisequidistributed}
|\Omega_{r,n_1,p_1,B_2}| < 2^r W^{-B_2-2} Hs^2.
\end{equation}
\end{proposition}
By the previous analysis and formulas (\ref{10/28formula4}) and (\ref{11/4formula1}), to prove Theorem \ref{minor arc estimate}, it remains to proving Proposition \ref{small portion with not equidistributed}. Next, we explain the main issue to prove Proposition \ref{small portion with not equidistributed}. In the proof,
we mainly apply the method of \cite[Section 7]{HW}, in which the authors deal with $g'(n,pl-n)$, a polynomial sequence of variables $n,p,l$ in general nilmanifolds. Specially $g'(n,pl-n)$ is a polynomial of variables $p,l$ in $\mathbb{R}/\mathbb{Z}$ if $g'(n,h)=f(n+h)$ for some polynomial $f:\mathbb{Z}\rightarrow \mathbb{R}$. While for our case, it is not enough if we just treat $g'(n,p(sl+a)-n)$ as a polynomial of variables $p,l$. Instead, we treat $g'(n,p(sl+a)-n)$ as a polynomial of variables $a,p,l$ because $a$ may vary in a wide range $[1,s]$ compared with $[1,H]$. Combining with this idea and the proof of \cite[Proposition 6.9]{HW}, under the assumption that (\ref{thecardinalityofdisequidistributed}) does not hold, we shall obtain a conclusion contradicting the property that the sequence $\{g'(n,h)\}_{h\in \mathcal{A}_{a}}$ is totally $W^{-B/2}$-equidistribution in $\mathbb{R}/\mathbb{Z}$ for almost all $a$ in $[1,s]$ (see property (ii) below formula (\ref{decomposition})). For self-containing, we give the detailed proof of Proposition \ref{small portion with not equidistributed} in Appendix \ref{aproofofaproposition}.

Finally, we are ready to prove Theorem \ref{the estimate of general multiplicative functions}.
\begin{proof}[Proof of Theorem \ref{the estimate of general multiplicative functions}]
By formulas (\ref{afternoon1}), (\ref{afternoon2}), (\ref{afternoon3}) and Theorems \ref{the estimate of major arc part}, \ref{minor arc estimate}, we obtain the statement in this theorem.
\end{proof}
\begin{remark}\label{another estimate about Theorem1}
{\rm In the estimate of the major arc case (see Subsection \ref{the estimate of major arc case}), if we do not split the summation over $h$ into $n+h\in \mathcal{S}$ and $n+h\notin\mathcal{S}$, we shall obtain the following estimate under the same assumptions as Theorem \ref{the estimate of general multiplicative functions},
\begin{align}\label{1/10formula1}
\begin{aligned}
&\sum_{a=1}^s \sum_{n=1}^{N}\left|\sum_{\substack{1\leq h\leq Hs\\n+h\equiv a(mod~s)}}\beta(n+h)F\big(P(n+h)(mod~\mathbb{Z})\big)\right|\ll NHs\Big(\frac{s}{\varphi(s)}\frac{\log\log H}{\log H}\Big)^{1/2}\\
&+ NHs\Big(\frac{\log\log H}{(\log H)^{1/2}}+\frac{1}{(\log N)^{1/600}}\Big)+NHs\frac{M\big(\beta;N/\log N,2N,(\log N)^{1/32}\big)+1}{\exp\Big(\frac{1}{2}M\big(\beta;N/\log N,2N,(\log N)^{1/32}\big)\Big)}.
\end{aligned}
\end{align}
Compared with (\ref{1/18}), it removes $(\log N)^{o(1)}$ before $M(\beta;\cdot,\cdot,\cdot)\exp(-\frac{1}{2}M(\beta;\cdot,\cdot,\cdot))$.}
\end{remark}
\section{Proof of Theorem \ref{disjointfromasymptoticallyperiodicfunctionsb}}\label{4}
To prove Theorem \ref{disjointfromasymptoticallyperiodicfunctionsb}, we need some preparation. The following result is on the self-correlation of $\mu(n)e(P(n))$ in short arithmetic progressions on average.
\begin{lemma}\label{the distribution of mobius in short arithmetic interval}
Let $s\geq 1$ and $h\geq 3$ be integers. Suppose that $P(x)\in \mathbb{R}[x]$ is of degree $d\geq 0$. Then
\begin{equation}\label{has relation with k}
\limsup_{N\rightarrow \infty}\frac{1}{N}\sum_{n=1}^{N}\Bigg|\frac{1}{h}\sum_{l=1}^{h}\mu(n+ls)e\big(P(n+ls)\big)\Bigg|^2\ll \frac{s}{\varphi(s)}\frac{\log\log h}{\log h},
\end{equation}
where the implied constant depends on $d$ at most.
\end{lemma}
\begin{proof}
Given $s\geq 1$ and $h\geq 3$, for $N$ large enough,
\begin{align*}
  &\sum_{n=1}^{N}\Bigg|\sum_{l=1}^{h}\mu(n+ls)e\big(P(n+ls)\big)\Bigg|=\sum_{a=1}^{s}\sum_{\substack{n=1\\n\equiv a(mod~s)}}^{N}\Bigg|\sum_{l=1}^{h}\mu(n+ls)e\big(P(n+ls)\big)\Bigg|\\
  &=\sum_{a=1}^{s}
  \sum_{m=1}^{N/s}\Bigg|\sum_{l=1}^{h}\mu(ms+ls+a)e\big(P(ms+ls+a)\big)\Bigg|+O(hs)\\
 &=\sum_{a=1}^{s}
  \sum_{m=1}^{N/s}\Bigg|\sum_{\substack{n=(m+1)s\\n\equiv a(mod~s)}}^{(m+h+1)s}\mu(n)e\big(P(n)\big)\Bigg|+O(hs)\\
&=\sum_{a=1}^{s}
  \sum_{\substack{x=1\\s|x}}^{N}\Bigg|\sum_{\substack{n=x\\n\equiv a(mod~s)}}^{x+hs}\mu(n)e\big(P(n)\big)\Bigg|+O(hs),
\end{align*}
then
\begin{equation}\label{an identity}
\sum_{n=1}^{N}\Bigg|\sum_{l=1}^{h}\mu(n+ls)e\big(P(n+ls)\big)\Bigg|=\frac{1}{s}\sum_{a=1}^{s}
  \sum_{x=1}^{N}\Bigg|\sum_{\substack{n=x\\n\equiv a(mod~s)}}^{x+hs}\mu(n)e\big(P(n)\big)\Bigg|+O(N).
\end{equation}
By the above and Theorem \ref{the theorem we should prove},
\begin{align*}
\limsup_{N\rightarrow \infty}\frac{1}{N}\sum_{n=1}^{N}\Bigg|\sum_{l=1}^{h}\mu(n+ls)e\big(P(n+ls)\big)\Bigg|^2&\leq h\limsup_{N\rightarrow \infty}\frac{1}{N}\sum_{n=1}^{N}\Bigg|\sum_{l=1}^{h}\mu(n+ls)e\big(P(n+ls)\big)\Bigg|\\
&\ll h^2\frac{s}{\varphi(s)}\frac{\log\log h}{\log h}
\end{align*}
as claimed.
\end{proof}
In the following, we prove a more general result than Theorem \ref{disjointfromasymptoticallyperiodicfunctionsb}.
\begin{lemma}\label{disjointfromasymptoticallyperiodicfunctions1b}
Let $(X,T)$ be a topological dynamical system. Suppose $x_{0}\in X$ satisfies the following condition: for any $\nu\in \mathcal{M}(x_{0};X,T)$, there is a dense set $\mathcal{F}\subseteq C(X)$ such that for each $g(x)\in\mathcal{F}$ we can find sequences $\{h_{j}\}_{j=1}^{\infty}$ and $\{s_{j}\}_{j=1}^{\infty}$ (may depend on $\nu$, $g$) of positive integers with
\[
\lim_{j\rightarrow \infty}\frac{\log\log h_{j}}{\log h_{j}}\frac{s_{j}}{\varphi(s_{j})}=0
\]
and
\[
\lim_{j\rightarrow \infty}\frac{1}{h_{j}}\sum_{l=1}^{h_{j}}\|g\circ T^{ls_{j}}-g\|_{L^2(\nu)}^2=0.
\]
Then for any $P(x)\in \mathbb{R}[x]$ and $f(x)\in C(X)$,
\begin{equation}\label{1/11formula2}
\lim_{N\rightarrow \infty}\frac{1}{N}\sum_{n=1}^{N}\mu(n)e(P(n))f(T^{n}x_{0})=0.\
\end{equation}
\end{lemma}
\begin{proof}
Assume on the contrary, there is a $P(x)\in \mathbb{R}[x]$ and $f(x)\in C(X)$ such that formula (\ref{1/11formula2}) does not hold,
then there is a constant $c_{0}>0$ and an increasing sequence $\{N_{m}\}_{m=1}^{\infty}$ of positive integers such that
\begin{equation}\label{two}
\frac{1}{N_{m}}\Bigg|\sum_{n=0}^{N_{m}-1}\mu(n)e(P(n))f(T^{n}x_{0})\Bigg|\geq 2c_{0}.
\end{equation}
Note that $X$ is a compact metric space. Then there is a subsequence of $\{N_{m}\}_{m=1}^{\infty}$ (denoted by $\{N_{m}\}_{m=1}^{\infty}$ again for convenience) and a $T$-invariant measure $\nu$ on $X$ such that
$\nu_{N_{m}}=\frac{1}{N_{m}}\sum_{n=0}^{N_{m}-1}\delta_{T^{n}x_{0}}$ weak* converges to $\nu$ as $m\rightarrow \infty$, i.e., for any $h(x)\in C(X)$,
\[\lim_{m\rightarrow \infty}\int_{X}h(x)d\nu_{N_{m}}=\lim_{m\rightarrow \infty}\frac{1}{N_{m}}\sum_{n=0}^{N_{m}-1}h(T^{n}x_{0})=\int_{X}h(x)d\nu.\]
By formula (\ref{two}) and the condition stated in this lemma, there is a $g(x)\in C(X)$ and sequences $\{h_{j}\}_{j=1}^{\infty}$, $\{s_{j}\}_{j=1}^{\infty}$ of positive integers satisfying
\begin{equation}\label{1/11noon}
\lim_{j\rightarrow \infty}\frac{\log\log h_{j}}{\log h_{j}}\frac{s_{j}}{\varphi(s_{j})}=0
\end{equation}
and
\begin{equation}\label{three}
\lim_{j\rightarrow\infty}\frac{1}{h_{j}}\sum_{l=1}^{h_{j}}\|g\circ T^{ls_{j}}-g\|_{L^2(\nu)}^2=0.
\end{equation}
Moreover, for $m=1,2,\ldots$,
\begin{equation}\label{four}
\frac{1}{N_{m}}\Bigg|\sum_{n=0}^{N_{m}-1}\mu(n)e(P(n))g(T^{n}x_{0})\Bigg|\geq c_{0}.
\end{equation}
Let $\widetilde{g}(n)=g(T^{n}x_{0})$. Note that
\begin{align*}
\lim_{m\rightarrow \infty}\frac{1}{N_{m}}\sum_{n=0}^{N_{m}-1}|\widetilde{g}(n+ls_{j})-\widetilde{g}(n)|^2=&\lim_{m\rightarrow \infty}\frac{1}{N_{m}}\sum_{n=0}^{N_{m}-1}|g(T^{ls_{j}+n}x_{0})-g(T^{n}x_{0})|^2\\
=&\|g\circ T^{ls_{j}}-g\|_{L^2(\nu)}^2.
\end{align*}
So by equation (\ref{three}),
\begin{equation}\label{seven}
\lim_{j\rightarrow\infty}\frac{1}{h_{j}}\sum_{l=1}^{h_{j}}\lim_{m\rightarrow \infty}\frac{1}{N_{m}}\sum_{n=0}^{N_{m}-1}|\widetilde{g}(n+ls_{j})-\widetilde{g}(n)|^2=0.
\end{equation}
Let $\delta=\frac{c_{0}}{2(\|\widetilde{g}\|_{l^{\infty}}+1)}$.
Then by Lemma \ref{the distribution of mobius in short arithmetic interval} and equations (\ref{1/11noon}), (\ref{seven}), there is a positive integer $j_{0}$ and a positive integer $m_{0}$ such that whenever $m>m_{0}$,
\[\frac{1}{N_{m}}\sum_{n=0}^{N_{m}-1}\Bigg|\frac{1}{h_{j_{0}}}\sum_{l=1}^{h_{j_{0}}}\mu(n+ls_{j_{0}})e\big(P(n+ls_{j_{0}})\big)\Bigg|^2<\delta^2/9\]
and
\[\frac{1}{h_{j_{0}}}\sum_{l=1}^{h_{j_{0}}}\frac{1}{N_{m}}\sum_{n=0}^{N_{m}-1}|\widetilde{g}(n+ls_{j_{0}})-\widetilde{g}(n)|^2<\delta^2/9.\]
Then by the Cauchy-Schwarz inequality,
\begin{equation}\label{0817formula31/11}
\frac{1}{N_{m}}\sum_{n=0}^{N_{m}-1}\Bigg|\frac{1}{h_{j_{0}}}\sum_{l=1}^{h_{j_{0}}}\mu(n+ls_{j_{0}})e\big(P(n+ls_{j_{0}})\big)\Bigg|<\delta/3
\end{equation}
and
\begin{equation}\label{0817formula21/11}
\frac{1}{h_{j_{0}}}\sum_{l=1}^{h_{j_{0}}}\frac{1}{N_{m}}\sum_{n=0}^{N_{m}-1}
|\widetilde{g}(n+ls_{j_{0}})-\widetilde{g}(n)|<\delta/3.
\end{equation}
Observe that
\[\frac{1}{N_{m}}\sum_{n=0}^{N_{m}-1}\mu(n)e\big(P(n)\big)\widetilde{g}(n)=\frac{1}{N_{m}}\sum_{n=0}^{N_{m}-1}\mu(n+ls_{j_{0}})e\big(P(n+ls_{j_{0}})\big)\widetilde{g}(n+ls_{j_{0}})
+O(\frac{ls_{j_{0}}}{N_{m}}).\]
Then there is a positive integer $m_{1}$ such that whenever $m>m_{1}$,
\begin{equation}\label{0817formula41/11}
\Bigg|\frac{1}{N_{m}}\sum_{n=0}^{N_{m}-1}\mu(n)e\big(P(n)\big)\widetilde{g}(n)-\frac{1}{N_{m}}\sum_{n=0}^{N_{m}-1}\frac{1}{h_{j_{0}}}\sum_{l=1}^{h_{j_{0}}}\mu(n+ls_{j_{0}})
e\big(P(n+ls_{j_{0}})\big)
\widetilde{g}(n+ls_{j_{0}})\Bigg|<\delta/3.
\end{equation}
Let $m_{2}=\max\{m_{0},m_{1}\}$. Then by formulas (\ref{0817formula31/11}), (\ref{0817formula21/11}) and (\ref{0817formula41/11}), for $m>m_{2}$, we have
\begin{align*}
\Bigg|\frac{1}{N_{m}}\sum_{n=0}^{N_{m}-1}\mu(n)e\big(P(n)\big)\widetilde{g}(n)\Bigg|<&\Bigg|\frac{1}{N_{m}}\sum_{n=0}^{N_{m}-1}\frac{1}{h_{j_{0}}}\sum_{l=1}^{h_{j_{0}}}
\mu(n+ls_{j_{0}})
e\big(P(n+ls_{j_{0}})\big)
\widetilde{g}(n+ls_{j_{0}})\Bigg|+\delta/3 \nonumber\\
\leq &\Bigg|\frac{1}{N_{m}}\sum_{n=0}^{N_{m}-1}\frac{1}{h_{j_{0}}}\sum_{l=1}^{h_{j_{0}}}\mu(n+ls_{j_{0}})
e\big(P(n+ls_{j_{0}})\big)
\big(\widetilde{g}(n+ls_{j_{0}})-\widetilde{g}(n)\big)\Bigg|\nonumber\\
&+\Bigg|\frac{1}{N_{m}}\sum_{n=0}^{N_{m}-1}\frac{1}{h_{j_{0}}}\sum_{l=1}^{h_{j_{0}}}\mu(n+ls_{j_{0}})
e\big(P(n+ls_{j_{0}})\big)\widetilde{g}(n)
\Bigg|+\delta/3 \nonumber\\
\leq &\frac{1}{N_{m}}\sum_{n=0}^{N_{m}-1}\frac{1}{h_{j_{0}}}\sum_{l=1}^{h_{j_{0}}}
|\widetilde{g}(n+ls_{j_{0}})-\widetilde{g}(n)|\nonumber\\&+\frac{\|\widetilde{g}\|_{l^{\infty}}}{N_{m}}\sum_{n=0}^{N_{m}-1}\Bigg|\frac{1}{h_{j_{0}}}\sum_{l=1}^{h_{j_{0}}}\mu(n+ls_{j_{0}})
e\big(P(n+ls_{j_{0}})\big)\Bigg|
+\delta/3\nonumber\\<&c_{0}/2.
\end{align*}
This contradicts formula (\ref{four}). Hence we obtain the claim in this lemma.
\end{proof}
\begin{remark}\label{remark to Theorem 1.6a}
{\rm With the help of Theorem \ref{the estimate of general multiplicative functions}, we see that the results listed in Lemmas \ref{the distribution of mobius in short arithmetic interval} and \ref{disjointfromasymptoticallyperiodicfunctions1b} also hold if $\mu$ is replaced by a more general ``non-pretentious'' 1-bounded  multiplicative function (such as the Liouville function and $\mu(n)\chi(n)$, where $\chi$ is a given Dirichlet character).}
\end{remark}
Note that functions realized in affine linear flows on compact abelian groups of zero topological entropy can be approximated by finite linear combinations of piece-wise polynomial phases (see \cite[Subsections 2.1, 2.2]{JP}).
As an application of Lemma \ref{disjointfromasymptoticallyperiodicfunctions1b},
we now prove Theorem \ref{disjointfromasymptoticallyperiodicfunctionsb}.
\begin{proof}[Proof of Theorem \ref{disjointfromasymptoticallyperiodicfunctionsb}]
In the following, we prove that for any $F\in C(X_{1}\times X_{2})$ and $x\in X_{1}\times X_{2}$,
\begin{equation}\label{equation for Proposition 1.8a1}
\lim_{N\rightarrow \infty}\frac{1}{N}\sum_{n=1}^{N}\mu(n)F(T^{n}x)=0,
\end{equation}
where $T=T_{1}\times T_{2}$. Let $\mathcal{S}$ be the *-subalgebra (i.e., subalgebra closed under the complex conjugation) of $C(X_{1}\times X_{2})$ generated by $\{\psi \cdot f:\psi$ is a character of $X_{1}$, $f\in C(X_{2})$ $\}$.
By the Stone-Weierstrass theorem (see, e.g., [KR83, Theorem 3.4.14]), $\mathcal{S}$ is dense in $C(X_{1}\times X_{2})$. So it suffices to show that for each $F\in C(X_{1}\times X_{2})$ in form of $\psi \cdot f$, equation (\ref{equation for Proposition 1.8a})  holds. Given $x=(x_{1},x_{2})$, there is a positive integer $w$ such that
\begin{equation}\label{0215formula1}
\frac{1}{N}\sum_{n\leq N}\mu(n)F(T^{n}x)=\frac{1}{N}\sum_{l=0}^{w-1}\sum_{\substack{n\leq N\\n\equiv l(mod~w)}}\mu(n)e(\phi(n))f(T_{2}^{n}x_{2}),
\end{equation}
where $\phi(n)$'s are some polynomials in $n$ with coefficients depending on $w$, $l$ \cite[equation (2.10)]{JP}.
By Remark \ref{remark to Theorem 1.6a}, Lemma \ref{disjointfromasymptoticallyperiodicfunctions1b} holds with $\mu(n)$ replaced by $\mu(n)1_{n\equiv l(mod~w)}$. Hence formula (\ref{0215formula1}) converges to zero as $N\rightarrow \infty$.
\end{proof}
\section{Proof of Theorem \ref{0707theorem1}}
In this section, we first show a sufficient condition to Problem \ref{sarnak's conjecture for rigid dynamical systems}, which states the validity of Sarnak's conjecture for product flows between affine linear flows on compact abelian groups of zero topological entropy and all rigid dynamical systems.

\begin{proposition}\label{a sufficient condition for problem2a}
Assume that for any given $P(x)\in \mathbb{R}[x]$,
\be  \label{prop for problem 2a}
\limsup_{N \rightarrow \infty} \frac{1}{Ns} \sum_{a=1}^s \sum_{n=1}^{N} \left|\sum_{\substack{ m=n+1 \\ m \equiv a(mod~s)} }^{n+hs} \mu(m) e(P(m)) \right| = o(h)
\ee
where the little ``o" term is independent of $s\geq 1$. Then the M\"obius Disjointness Conjecture holds for the product flow $(X_{1}\times X_{2},T_{1}\times T_{2})$, where $(X_{1},T_{1})$ is an affine linear flow on a compact abelian group of zero entropy and $(X_{2},T_{2})$ is a flow such that for any $x_{2}\in X_{2}$ and any $\nu\in \mathcal{M}(x_{2};X_{2},T_{2})$, $(X_{2},\nu,T_{2})$ is rigid.
\end{proposition}
\begin{proof}
Let $P(x)\in \mathbb{R}[x]$. Let $(X,T)$ be a flow satisfying that for any $x_{0}\in X$ and any $\nu\in \mathcal{M}(x_{0};X,T)$, $(X,\nu,T)$ is rigid. We claim that for any $x_{0}\in X$ and $g\in C(X)$,
\begin{equation}\label{115formula2}
\lim_{N\rightarrow \infty}\frac{1}{N}\sum_{n=1}^{N}\mu(n)e(P(n))g(T^{n}x_{0})=0.
\end{equation}
Suppose that the sequence $\frac{1}{N_{m}}\sum_{n=0}^{N_{m}-1}\delta_{T^{n}x_{0}}$ weak* converges to a Borel probability measure $\nu$ on $X$. Let $\widetilde{g}(n)=g(T^{n}x_{0})$. For any given $\epsilon>0$, let $\delta=\frac{\epsilon}{3(\|\widetilde{g}\|_{l^{\infty}}+1)}$.
Assume that equation (\ref{prop for problem 2a}) holds, by equation (\ref{an identity}), we have that
\[\limsup_{N\rightarrow \infty}\frac{1}{N}\sum_{n=1}^{N}\Bigg|\frac{1}{h}\sum_{l=1}^{h}\mu(n+ls)e\big(P(n+ls)\big)\Bigg|\]
tends to zero uniformly with respect to $s$ as $h\rightarrow \infty$.  By the above and the rigidity of $(X,\nu,T)$, there is a positive integer $h_{0}$ such that
\begin{equation}\label{12/10formula1}
\limsup_{N\rightarrow \infty}\frac{1}{N}\sum_{n=1}^{N}\Bigg|\frac{1}{h_{0}}\sum_{l=1}^{h_{0}}\mu(n+ls)e\big(P(n+ls)\big)\Bigg|<\delta
\end{equation}
holds for any $s\geq 1$, and there is a positive integer $s_{j_{0}}$ such that
\[\|g\circ T^{s_{j_{0}}}-g\|_{L^2(\nu)}^2=\lim_{m\rightarrow \infty}\frac{1}{N_{m}}\sum_{n=1}^{N_{m}}|\widetilde{g}(n+s_{j_{0}})-\widetilde{g}(n)|^2<\Big(\frac{2\delta}{h_{0}+1}\Big)^2.\]
Then by the Cauchy-Schwarz inequality,
\begin{equation}\label{12/10afternoon}
\limsup_{m\rightarrow \infty}\frac{1}{N_{m}}\sum_{n=1}^{N_{m}}|\widetilde{g}(n+s_{j_{0}})-\widetilde{g}(n)|<\frac{2\delta}{h_{0}+1}.
\end{equation}
By formula (\ref{12/10afternoon}) and the triangle inequality,
\begin{align}\label{12/10formula2}
\begin{aligned}
&\limsup_{m\rightarrow \infty}\frac{1}{N_{m}}\sum_{n=1}^{N_{m}}\frac{1}{h_{0}}\sum_{l=1}^{h_{0}}
|\widetilde{g}(n+ls_{j_{0}})-\widetilde{g}(n)|\leq \frac{1}{h_{0}}\sum_{l=1}^{h_{0}}\limsup_{m\rightarrow \infty}\frac{1}{N_{m}}\sum_{n=1}^{N_{m}}|\widetilde{g}(n+ls_{j_{0}})-\widetilde{g}(n)|\\
&\leq\frac{1}{h_{0}}\sum_{l=1}^{h_{0}}\sum_{i=1}^{l}\limsup_{m\rightarrow \infty}\frac{1}{N_{m}}\sum_{n=1}^{N_{m}}|\widetilde{g}(n+is_{j_{0}})-\widetilde{g}(n+(i-1)s_{j_{0}})|\\
&=\frac{1}{h_{0}}\sum_{l=1}^{h_{0}}\sum_{i=1}^{l}\limsup_{m\rightarrow \infty}\frac{1}{N_{m}}\sum_{n=1}^{N_{m}}|\widetilde{g}(n+s_{j_{0}})-\widetilde{g}(n)|<\delta.
\end{aligned}
\end{align}
We apply an argument similar to the proof of Lemma \ref{disjointfromasymptoticallyperiodicfunctions1b} with formulas (\ref{0817formula31/11}) and (\ref{0817formula21/11}) are replaced by (\ref{12/10formula1}) and (\ref{12/10formula2}), respectively, then
\[\limsup_{m\rightarrow \infty}\Bigg|\frac{1}{N_{m}}\sum_{n=1}^{N_{m}}\mu(n)e\big(P(n)\big)g(T^{n}x_{0})\Bigg|<\epsilon.\]
Letting $\epsilon\rightarrow 0$, we have
$\lim_{m\rightarrow \infty}\frac{1}{N_{m}}\sum_{n=1}^{N_{m}}\mu(n)e\big(P(n)\big)g(T^{n}x_{0})=0$. Since this equality holds for any sequence $\{N_{m}\}_{m}$ such that $\frac{1}{N_{m}}\sum_{n=0}^{N_{m}-1}\delta_{T^{n}x_{0}}$ weak* converges, equation (\ref{115formula2}) holds. It is obvious that the indicator function $1_{n\equiv l(mod~w)}$ is periodic for any integer $l$ and any positive integer $w$. Then $1_{n\equiv l(mod~w)}$ is a finite combination of $e(\frac{na}{w})$ for $a=1,2,\ldots,w$. So equation (\ref{115formula2}) also holds with $\mu(n)$ replaced by $\mu(n)1_{n\equiv l(mod~w)}$. By equations (\ref{0215formula1}) and (\ref{115formula2}), we obtain the claim in this proposition.
\end{proof}
\begin{remark}\label{related to Chowla's conjecture}{\rm
When $P(x)$ is a constant, condition (\ref{a sufficient condition for problem2a}) in the above proposition is implied by Chowla's conjecture. In fact, under the Chowla conjecture, for integer $s\geq 1$, \[\limsup_{N \rightarrow \infty} \frac{1}{N} \sum_{n=1}^{N} \left|\sum_{l=1}^h \mu(n+ls)\right|^2=6h/\pi^2.\] Note that
\[
\frac{1}{Ns} \sum_{a=1}^s \sum_{n=1}^{N} \left|\sum_{\substack{m=n+1\\ m \equiv a(mod~s)} }^{n+hs} \mu(m)\right|=\frac{1}{N} \sum_{n=1}^{N} \left|\sum_{l=1}^h \mu(n+ls)\right| + O(1).
\]
So by the Cauchy-Schwarz inequality, the left-hand side of the above equation $\leq c\sqrt{h}$,
where $c$ is an absolute constant. It is obvious that this bound is independent of $s$.
}\end{remark}

In the following, we shall prove Theorem \ref{0707theorem1}, which gives a positive answer to Problem \ref{sarnak's conjecture for rigid dynamical systems} under the logarithmic average. Let $X$ be a compact metric space, $T:X\rightarrow X$ a continuous map and $x_{0}\in X$. Then $C(X)$ is countably generated as a C*-algebra \cite[Remark 3.4.15]{KR}.  For any integer $N\geq 1$, consider the bounded linear functional $\frac{1}{\log N}\sum_{n=1}^{N}\frac{1}{n}\delta_{T^{n}x_{0}}: f\mapsto \frac{1}{\log N}\sum_{n=1}^{N}\frac{1}{n}f(T^nx_{0})$ for any $f\in C(X)$. Note that the norms of these functionals are uniformly bounded with respect to $N$, by Alaoglu-Bourbaki theorem, the bounded ball in $C(X)^{\sharp}$, the set of all bounded linear functionals on $C(X)$, is weak*-compact. Moreover, the bounded ball is also metrizable since $C(X)$ is countably generated. So there exist (weak*) limit points of $\{\frac{1}{\log N}\sum_{n=1}^{N}\frac{1}{n}\delta_{T^{n}x_{0}}\}_{N=2}^{\infty}$. For a limit point along the subsequence $\{N_{m}\}_{m=1}^{\infty}$ in $C(X)^{\sharp}$, denoted by $\rho$, by the Riesz representation theorem and \cite[Theorem 6.8]{wal}, it is not hard to show that there is a unique $T$-invariant Borel probability measure $\nu$ on $X$ such that for any $f(x)\in C(X)$,
\begin{equation}\label{0708formula1}
\rho(f)=\lim_{m\rightarrow \infty}\frac{1}{\log N_{m}}\sum_{n=1}^{N_{m}}\frac{f(T^{n}x_{0})}{n}=\int_{X}f(x)d\nu(x).
\end{equation}

The next one is a general version of Theorem \ref{0707theorem1}.

\begin{proposition}\label{0706afternoonformula8}
Let $(X,T)$ be a topological dynamical system such that for any $x_{0}\in X$ and any $\nu$ in the weak*
closure of $\{\frac{1}{\log N}\sum_{n=1}^{N}\frac{1}{n}\delta_{T^{n}x_{0}}: N=2,\ldots\}$ in the space of Borel measures\footnote{Though $\frac{1}{\log N}\sum_{n=1}^{N}\frac{1}{n}\delta_{T^{n}x_{0}}$ is not a probability measure, it is not hard to show that $\nu$ is a $T$-invariant Borel probability measure on X.} on $X$, $(X,\nu,T)$ is rigid. Then the logarithmically averaged M\"obius Disjointness Conjecture holds for the product flow between $(X,T)$ and any affine linear flow on a compact abelian group of zero entropy.
\end{proposition}
\begin{proof}
Let $P(x)\in \mathbb{R}[x]$. We claim that for any $x_{0}\in X$ and $g\in C(X)$,
\begin{equation}\label{20220706formula2}
\lim_{N\rightarrow \infty}\frac{1}{N}\sum_{n=1}^{N}\frac{\mu(n)e(P(n))g(T^{n}x_{0})}{n}=0.
\end{equation}
Assume the contrary that formula (\ref{20220706formula2}) does not hold. Then there is a $\delta>0$ and a subsequence $\{N_{m}\}_{m=1}^{\infty}$ of natural numbers such that
\begin{equation}\label{706formula1}
\frac{1}{\log N_{m}} \Bigg|\sum_{n=1}^{N_{m}} \frac{\mu(n) e(P(n))g(T^nx_{0})}{n} \Bigg|>\delta.
\end{equation}
By (\ref{0708formula1}), there is a subsequence of $\{N_{m}\}_{m=1}^{\infty}$ (denoted by $\{N_{m}\}_{m=1}^{\infty}$ again for convenience) and a $T$-invariant Borel probability measure $\nu$ on $X$, such that
$\nu_{N_{m}}=\frac{1}{\log N_{m}}\sum_{n=1}^{N_{m}}\frac{1}{n}\delta_{T^{n}x_{0}}$ weak* converges to $\nu$ as $m\rightarrow \infty$, i.e., for any $f\in C(X)$,
\[\lim_{m\rightarrow \infty}\int_{X}f(x)d\nu_{N_{m}}=\lim_{m\rightarrow \infty}\frac{1}{\log N_{m}}\sum_{n=1}^{N_{m}}\frac{f(T^{n}x_{0})}{n}=\int_{X}f(x)d\nu.\]
Let $\widetilde{g}(n)=g(T^{n}x_{0})$. Then by the condition in this proposition and the above equation, there is a sequence $\{n_{j}\}_{j=1}^{\infty}$ of positive integers such that
\begin{equation}\label{0706afternoonformula1}
\lim_{j\rightarrow \infty}\lim_{m\rightarrow \infty}\frac{1}{\log N_{m}}\sum_{n=1}^{N_{m}}\frac{1}{n}|\widetilde{g}(n+n_{j})-\widetilde{g}(n)|^2=\lim_{j\rightarrow\infty}\int_{X}|g\circ T^{n_{j}}(x)-g(x)|^2d\nu(x)=0.
\end{equation}
On the other hand, by the remark after \cite[Corollary 1.5]{FH} and \cite[Theorem 2]{Tao16}, we have
\begin{equation}\label{215formula41}
\lim_{N \rightarrow \infty} \frac{1}{\log N} \sum_{n=1}^{N} \frac{\mu(n+h_1) \mu(n+h_2)e(P(n)) }{n}= 0
\end{equation}
for any $h_1 \neq h_2 \in \mathbb{N}$.
Hence
\begin{align}
\begin{aligned}
&\lim_{N\rightarrow \infty}\frac{1}{\log N}\sum_{n=1}^{N}\frac{1}{n}\Bigg|\frac{1}{h}\sum_{l=1}^{h}\mu(n+ls)e\big(P(n+ls)\big)\Bigg|^2\\
&=\frac{1}{h^2}\sum_{l_{1}=1}^{h}\sum_{l_{2}=1}^{h}\lim_{N\rightarrow \infty}\frac{1}{\log N}\sum_{n=1}^{N}\frac{1}{n}\mu(n+l_{1}s)\mu(n+l_{2}s)e\big(P(n+l_{1}s)-P(n+l_{2}s)\big)\\
&=\frac{1}{h^2}\sum_{l=1}^{h}\lim_{N\rightarrow \infty}\frac{1}{\log N}\sum_{n=1}^{N}\frac{\mu^2(n+ls)}{n}+\frac{1}{h^2}\sum_{l_{1}=1}^{h}\sum_{\substack{l_{2}=1\\l_{2}\neq l_{1}}}^{h}\lim_{N\rightarrow \infty}\frac{1}{\log N}\sum_{n=1}^{N}\frac{\mu(n+l_{1}s)\mu(n+l_{2}s)e(Q(n))}{n}\\
&=\frac{6}{\pi^2}\frac{1}{h},
\end{aligned}
\end{align}
where $Q(n)=P(n+l_{1}s)-P(n+l_{2}s)$ is some polynomial with coefficients depending on $l_{1},l_{2},s$.
So by the Cauchy-Schwarz inequality,
\begin{equation}\label{1212formula41}
\limsup_{N\rightarrow \infty}\frac{1}{\log N}\sum_{n=1}^{N}\frac{1}{n}\Bigg|\frac{1}{h}\sum_{l=1}^{h}\mu(n+ls)e\big(P(n+ls)\big)\Bigg|=o_{h}(1),
\end{equation}
where the ``$o$" term is independent of $s=1,2,\ldots$.
Note that for any bounded function $\widetilde{f}(n)$,
\begin{equation}\label{invariant2}
\frac{1}{\log N}\sum_{n=1}^{N}\frac{\widetilde{f}(n)}{n}=\frac{1}{\log N}\sum_{n=1}^{N}\frac{\widetilde{f}(n+m)}{n+m}+O(\frac{m}{\log N})=\frac{1}{\log N}\sum_{n=1}^{N}\frac{\widetilde{f}(n+m)}{n}+O(\frac{m}{\log N}).
\end{equation}
Then by equations (\ref{0706afternoonformula1}), (\ref{1212formula41}), (\ref{invariant2}), and an almost identical proof to Proposition \ref{a sufficient condition for problem2a} with a slight difference by replacing the standard average form by the logarithmic average, we obtain (\ref{20220706formula2}) and so the claim in this proposition.
\end{proof}

\begin{proof}[Proof of Theorem \ref{0707theorem1}]
In order to apply Proposition \ref{0706afternoonformula8} to prove this theorem, we just need to check that for any $\nu$ in the weak*
closure of $\{\frac{1}{\log N}\sum_{n=1}^{N}\frac{1}{n}\delta_{T^{n}x_{0}}: N=1,2,\ldots\}$, $(X,T,\nu)$ is rigid. This comes from the fact that $\nu$ is a $T$-invariant Borel probability measure on $X$ and the condition in this theorem.
\end{proof}

As an application of Theorem \ref{0707theorem1}, we prove Corollary \ref{0707corollary1}.
\begin{proof}[Proof of Corollary \ref{0707corollary1}]
By lifting $\psi: \mathbb{T}^{1}\rightarrow \mathbb{T}^{1}$ to the real line, we can write $\psi(x)=cx+\psi_{1}(x)$ for all $x\in \mathbb{T}^{1}$, where $c\in \mathbb{Z}$ and $\psi_{1}:\mathbb{T}^{1}\rightarrow \mathbb{R}$ is a continuous 1-periodic function. Note that the set of linear combinations of characters on $\mathbb{T}^2$ are dense in $C(\mathbb{T}^2)$. Together with (\ref{0215formula1}), it is sufficient to prove that for any fixed $P(x)\in \mathbb{R}[x]$, $(x_{1},x_{2})\in X$ and $b=(b_{1},b_{2})\in \mathbb{Z}^2$,
\begin{equation}\label{0911formula1}
\lim_{N\rightarrow \infty}\frac{1}{\log N}\sum_{n=1}^{N}\frac{\mu(n)e(P(n))e\big(\langle b, T_{\alpha, \psi}^{n}(x_{1},x_{2})\rangle\big)}{n}=0.
\end{equation}
Suppose that $T_{\alpha, \psi}^{n}(x_{1},x_{2})=(y_{1}(n),y_{2}(n))$, where $y_{1}(n)=x_{1}+n\alpha$ and $y_{2}(n)=c\frac{n(n-1)}{2}\alpha+cnx_{1}+x_{2}+\sum_{j=0}^{n-1}\psi_{1}(x_{1}+j\alpha)$. It follows that
\[
\langle b, T_{\alpha, \psi}^{n}(x_{1},x_{2})\rangle=b_{1}y_{1}(n)+b_{2}y_{2}(n)=Q(n)+b_{1}y_{1}(n)+b_{2}(x_{2}+\sum_{j=0}^{n-1}\psi_{1}(x_{1}+j\alpha)),
\]
where $Q(n)=b_{2}(c\frac{n(n-1)}{2}\alpha+cnx_{1})$. Let $\widetilde{P}(n)=P(n)+Q(n)$. We rewrite (\ref{0911formula1}) as
\begin{equation}\label{0911formula3}
\lim_{N\rightarrow \infty}\frac{1}{\log N}\sum_{n=1}^{N}\frac{\mu(n)e(\widetilde{P}(n))e\big(\langle b, T_{\alpha, \psi_{1}}^{n}(x_{1},x_{2})\rangle\big)}{n}=0.
\end{equation}
For the case $\alpha\in \mathbb{Q}$, an argument similar to that of \cite[Proposition 3.1]{JP} leads to (\ref{0911formula3}).  For $\alpha\in \mathbb{R}\backslash \mathbb{Q}$,
it follows from \cite[Section 6.1]{Fav} that $(\mathbb{T}^2,T_{\alpha,\psi_{1}},\nu)$ is rigid for any $\nu$, a $T_{\alpha,\psi_{1}}$-invariant Borel probability measure on $\mathbb{T}^2$. By Theorem \ref{0707theorem1}, we have equation (\ref{0911formula3}).
\end{proof}

\bigskip
\appendix
\section{Proof of Proposition \ref{small portion with not equidistributed}}\label{aproofofaproposition}
In this section, we first show that the proof of formula (\ref{10/28formula4}) can be reduced to proving Proposition \ref{small portion with not equidistributed}. Then we give a proof of Proposition \ref{small portion with not equidistributed}. These results were used in Section \ref{the third main result}.
\begin{repproposition}{small portion with not equidistributed}
Let $C_{0}$ be a large enough constant only depending on $d$, $B_{2}\geq 10$, $B\geq 10C_{0}^2B_{2}$, $H\geq \max(W^{B},2^{10r_{+}})$. Let $s\geq 1$, $j=0,\ldots,q-1$ and $t=1,\ldots,W^2$. Let $r_{-}$, $r_{+}$ be as in formula (\ref{definition of P and Q}). For $r\in (r_{-},r_{+}]$, $n_{1}\in \mathcal{N}$ and a prime number $p_{1}\in (2^{r-1},2^{r}]$ with $(p_{1},s)=1$,
the set $\Omega_{r,n_1,p_1,B_2}$ is defined to be the set of all triples:
\[
(a,n,p) \in [1,s] \times \mathcal{N}\times \{p:p\in (2^{r-1},2^{r}],(p,s)=1\}.
\]
such that

(\romannumeral1) $p$ is prime;

(\romannumeral2) $|\mathcal{A}_{n_1,n,p_1,p,t} |\geq 2^{-r} W^{-(B_2+2)} Hs $;

(\romannumeral3)
\[\Bigg|\sum_{\substack{l\in \mathcal{A}_{n_{1},n,p_{1},p,t} \\ l\equiv js+a(qs)}}\widetilde{F}_{p_{1}a,n_{1},\textbf{j}_{1},t}(g'(n_{1},p_{1}l-n_1)\Gamma')
\overline{\widetilde{F}_{pa,n,\textbf{j},t}}(g'(n,pl-n)\Gamma') \Bigg|\geq \frac{W^{-B_2}|\mathcal{A}_{n_{1},n,p_{1},p,t}|}{qs}.\]
Then
\[|\Omega_{r,n_1,p_1,B_2}| < 2^r W^{-B_2-2} Hs^2.\]
\end{repproposition}
We first show that Proposition \ref{small portion with not equidistributed} implies formula (\ref{10/28formula4}).
\begin{proof}[Proof of inequality (\ref{10/28formula4}) (Assume Proposition \ref{small portion with not equidistributed})]
We first decompose (\ref{10/28formula5}) into three parts $S_{1}+S_{2}+S_{3}$ defined below. Then we estimate each term.

By formula (\ref{11/7morning}),
\begin{align*}
 S_{1}:= &\sum_{j=0}^{q-1}\sum_{t=1}^{W^2}\sum_{a=1}^s\sum_{ \substack{ p_1,p_2 \in(2^{r-1}, 2^r] \\ (p_1p_2, s) =1}}\sum_{\substack{n_{1}, n_{2}\in \mathcal{N}\\|\mathcal{A}_{n_1,n_{2},p_1,p_{2},t} |<2^{-r} W^{-(B_2+2)} Hs}}
 \\
&\Bigg|\sum_{\substack{l\in \mathcal{A}_{n_{1},n_{2},p_{1},p_2,t} \\ l\equiv js+a(qs)}}\widetilde{F}_{p_{1}a,n_{1},\textbf{j}_{1},t}(g'(n_{1},p_{1}l-n_1)\Gamma')
\overline{\widetilde{F}_{p_{2}a,n_{2},\textbf{j}_{2},t}}(g'(n_{2},p_{2}l-n_2)\Gamma') \Bigg|\\
&\ll sqW^2\sum_{n_{1}\in \mathcal{N}}\sum_{\substack{p_{1},p_{2}\in (2^{r-1},2^{r}]\\(p_{1}p_{2},s)=1}}\sum_{\substack{n_{2}\in \mathcal{N}\\ \mathcal{A}_{n_1,n_{2},p_1,p_{2},t}\neq \emptyset}}(|\mathcal{A}_{n_1,n_{2},p_1,p_{2},t}|/qs)\\
&\ll2^r W^{-(B_2+2)} H^2 N s^2.
\end{align*}

By Proposition \ref{small portion with not equidistributed},
\begin{align*}
 S_{2}:=&\sum_{j=0}^{q-1}\sum_{t=1}^{W^2}\sum_{n_{1}\in \mathcal{N}}\sum_{\substack{p_{1}\in (2^{r-1},2^{r}]\\(p_{1},s)=1}}\sum_{\substack{a,n_{2},p_{2}\\(a,n_{2},p_{2})\in \Omega_{r,n_1,p_1,B_2}\\|\mathcal{A}_{n_1,n_{2},p_1,p_{2},t} |\geq 2^{-r} W^{-(B_2+2)} Hs}}
\\
&\Bigg|\sum_{\substack{l\in \mathcal{A}_{n_{1},n_{2},p_{1},p_2,t} \\ l\equiv js+a(qs)}}\widetilde{F}_{p_{1}a,n_{1},\textbf{j}_{1},t}(g'(n_{1},p_{1}l-n_1)\Gamma')
\overline{\widetilde{F}_{p_{2}a,n_{2},\textbf{j}_{2},t}}(g'(n_{2},p_{2}l-n_2)\Gamma') \Bigg|\\
\ll & qW^2(2^{-r}W^{-2}Hs/qs)N2^{r}( 2^r W^{-B_2-2} Hs^2)\ll 2^{r}NH^2s^2W^{-(B_{2}+2)}.
\end{align*}

By the definition of $\Omega_{r,n_1,p_1,B_2}$ and formula (\ref{11/7morning}),
\begin{align*}\label{11/3formula13}
S_{3}:=&\sum_{j=0}^{q-1}\sum_{t=1}^{W^2}\sum_{n_{1}\in \mathcal{N}}\sum_{\substack{p_{1}\in (2^{r-1},2^{r}]\\(p_{1},s)=1}}\sum_{\substack{a,n_{2},p_{2}\\(a,n_{2},p_{2})\notin \Omega_{r,n_1,p_1,B_2}\\|\mathcal{A}_{n_1,n_{2},p_1,p_{2},t} |\geq 2^{-r} W^{-(B_2+2)} Hs}}\\
&\Bigg|\sum_{\substack{l\in \mathcal{A}_{n_{1},n_{2},p_{1},p_2,t} \\ l\equiv js+a(qs)}}\widetilde{F}_{p_{1}a,n_{1},\textbf{j}_{1},t}(g'(n_{1},p_{1}l-n_1)\Gamma')
\overline{\widetilde{F}_{p_{2}a,n_{2},\textbf{j}_{2},t}}(g'(n_{2},p_{2}l-n_2)\Gamma') \Bigg|\\
\ll&\sum_{j=0}^{q-1}\sum_{t=1}^{W^2}s\sum_{n_{1}\in \mathcal{N}}\sum_{\substack{p_{1},p_{2}\in (2^{r-1},2^{r}]\\(p_{1}p_{2},s)=1}}\sum_{\substack{n_{2}\in \mathcal{N}\\ \mathcal{A}_{n_1,n_{2},p_1,p_{2},t}\neq \emptyset}}\frac{W^{-B_2}|\mathcal{A}_{n_{1},n,p_{1},p,t}|}{qs}\\
\ll& NH^2s^22^{r}W^{-B_{2}-2}.
\end{align*}
Then we have \[S\leq (\ref{10/28formula5})=S_{1}+S_{2}+S_{3}\ll 2^{r}NH^2s^2W^{-(B_{2}+2)},\]
which is claimed in formula (\ref{10/28formula4}).
\end{proof}
We now prove Proposition \ref{small portion with not equidistributed}.
\begin{proof}[Proof of Proposition \ref{small portion with not equidistributed}]
For any given $j, t, p_1, n_1$, assume the contrary that \[|\Omega_{r,n_1,p_1,B_2}|\geq 2^r W^{-B_2-2} Hs^2.\]
Write $\mathcal{A}_{n_{1},p_{1},t}=[m_{1},m_{2}]$ with $m_{1}<m_{2}$ positive integers and $|m_{1}-m_{2}|\leq 2\cdot2^{-r}W^{-2}Hs$. Let $c_{0}=qs(\lfloor\frac{m_{1}}{qs}\rfloor-2)+js$.
Then $\{l:l\in \mathcal{A}_{n_{1},p_{1},t},l\equiv js+a(qs)\}\subseteq \{qsl+c_{0}+a:l\in [L]\}$, where $L=\lfloor\frac{m_{2}}{qs}\rfloor-\lfloor\frac{m_{1}}{qs}\rfloor+3\leq 4\cdot2^{-r}W^{-3}H+4$. Let $(a,n,p)\in \Omega_{r,n_1,p_1,B_2}$. By conditions $(\romannumeral2)$ and $(\romannumeral3)$, we have
\begin{equation}\label{10/31eveningformula8}
\Bigg|\sum_{l\in \mathcal{A}'_{n,p}}\widetilde{F}_{p_{1}a,n_{1},\textbf{j}_{1},t}(g'(n_{1},p_{1}(qsl+c_{0}+a)-n_1)\Gamma')
\overline{\widetilde{F}_{pa,n,\textbf{j},t}}(g'(n,p(qsl+c_{0}+a)-n)\Gamma') \Bigg|\geq W^{-B_2}|\mathcal{A}'_{n,p}|,
\end{equation}
where $\mathcal{A}'_{n,p}$ is a subinterval of $[L]$ with length $\lfloor\frac{|\mathcal{A}_{n_{1},n,p_{1},p,t}|}{qs}\rfloor$ or $\lfloor\frac{|\mathcal{A}_{n_{1},n,p_{1},p,t}|}{qs}\rfloor+1$. In the following, we apply the method used in \cite[Section 7]{HW} to our situation with appropriate modifications.

Denote by $g_{a,n_{1},p_{1}}(u)=g'(n_{1},qsu+p_{1}(c_{0}+a)-n_1)$ and $g_{a,n,p}(u)=g'(n,qsu+p(c_{0}+a)-n)$. By formula (\ref{10/31eveningformula8}), the sequence $\{(g_{a,n_{1},p_{1}}(p_{1}l)\Gamma',g_{a,n,p}(pl)\Gamma')\}_{l\in \mathcal{A}'_{n,p}}$ is not $2^{-2}W^{-B_{2}}$-equidistributed in $G'/\Gamma'\times G'/\Gamma'$. By \cite[Lemma 2.10]{HW}, we can find a short length $L'_{n,p}\geq 2^{-5}W^{-2B_{2}}L$ such that the sequence $\{(g_{a,n_{1},p_{1}}(p_{1}l)\Gamma',g_{a,n,p}(pl)\Gamma')\}_{l\in [L'_{n,p}]}$ is not $2^{-5}W^{-2B_{2}}$-equidistributed in $G'/\Gamma'\times G'/\Gamma'$. Then by Lemma \ref{Weyl rule}, there is a character $\eta_{a,n,p}\in \mathbb{Z}^{2}$ of $G'/\Gamma'\times G'/\Gamma'$ with $0<|\eta_{a,n,p}|<W^{O(B_{2})}$ such that
\begin{equation}\label{10/31eveningformula1}
\|\eta_{a,n,p}\circ (g_{a,n_{1},p_{1}}(p_{1}l)\Gamma',g_{a,n,p}(pl)\Gamma')\|_{C^{\infty}[L'_{n,p}]}\ll W^{O(B_{2})}.
\end{equation}
As $L'_{n,p}\gg W^{-2B_{2}}L$, this implies that
\begin{equation}\label{10/31eveningformula2}
\|\eta_{a,n,p}\circ (g_{a,n_{1},p_{1}}(p_{1}l)\Gamma',g_{a,n,p}(pl)\Gamma')\|_{C^{\infty}[L]}\ll W^{O(B_{2})}.
\end{equation}
By the pigeonhole principle, we can find a non-zero additive character $\eta\in \mathbb{Z}^{2}$ of $G'/\Gamma'\times G'/\Gamma'$ such that for a set $\Omega_{r,n_1,p_1,B_2}^{*}\subseteq \Omega_{r,n_1,p_1,B_2}$ with
\begin{equation}\label{12/06formula}
|\Omega_{r,n_1,p_1,B_2}^{*}|\geq 2^{r}W^{-O(B_{2})}Hs^2,
\end{equation}
formula (\ref{10/31eveningformula2}) holds and $\eta_{a,n,p}=\eta$.
Let $\mathcal{P}_{r,n_{1},p_{1}}=\{p: 2^{r-1}< p\leq 2^{r},(p,s)=1$, there are at least $W^{-O(B_{2})}Hs^2$ choices of $(a,n)$ such that $(a,n,p)\in \Omega_{r,n_1,p_1,B_2}^{*} \}$. Since
\[\Omega_{r,n_1,p_1,B_2}^{*}\leq |\mathcal{P}_{r,n_{1},p_{1}}|\cdot s \cdot|\{n: n\in \mathcal{N},\mathcal{A}_{n_1,n,p_1,p,t}\neq \emptyset \}|+2^{r-1}W^{-O(B_{2})}Hs^2,\]
we have
\begin{equation}\label{10/31formula4}
|\mathcal{P}_{r,n_{1},p_{1}}|\gg 2^{r}W^{-O(B_{2})}.
\end{equation}
Furthermore, for any $p\in \mathcal{P}_{r,n_{1},p_{1}}$, there are at least $W^{-O(B_{2})}s$ numbers of $a\in [s]$ such that $(a,n,p)\in \Omega_{r,n_1,p_1,B_2}^{*}$ for some $n$ by formula (\ref{11/7morning}).
By Remark \ref{form of g'}, we may suppose that $g'(n,h)=\sum_{i=0}^{d'}\alpha_{i}'(n+h)^{i}$ for some $d'$ with $1\leq d'\leq d$. Write $\eta=\eta_{1}\oplus\eta_{2}$. Due to the observation that the space of Lipschitz functions on $\mathbb{R}/\mathbb{Z}$ is essentially spanned by the space of pure phase functions $e(m\theta)$ and by a similar argument to the proof of \cite[Proposition 3.1]{GT2}, we have $\eta_{1}\neq 0$ and $\eta_{2}\neq 0$.
Suppose
\[\eta_{1}\circ (g_{a,n_{1},p_{1}}(p_{1}l)\Gamma')=\sum_{\substack{i_{1},i_{2}\geq 0\\i_{1}+i_{2}\leq d'}}\gamma_{i_{1},i_{2}}a^{i_{1}}l^{i_{2}}.\]
and
\begin{equation}\label{10/31formula1}
\eta_{2}\circ (g_{a,n,p}(u)\Gamma')=\sum_{\substack{l_{1},l_{2},l_{1}'\geq 0\\l_{1}+l_{2}\leq d',l_{1}'\leq l_{1}}}\beta_{l_{1},l_{2},l_{1}'}p^{l_{1}}a^{l_{1}'}u^{l_{2}}.
\end{equation}
Then
\begin{equation}\label{10/31formula2}
\eta\circ (g_{a,n_{1},p_{1}}(p_{1}l)\Gamma',g_{a,n,p}(pl)\Gamma')=\sum_{\substack{i_{1},i_{2}\geq 0\\i_{1}+i_{2}\leq d'}}\gamma_{i_{1},i_{2}}a^{i_{1}}l^{i_{2}}+\sum_{i=0}^{d'}\sum_{w=0}^{d'-i}\sum_{w'=0}^{w}\beta_{w,i,w'}p^{i+w}l^{i}a^{w'}.
\end{equation}

Now we treat the above as a polynomial in $l$. By formula (\ref{10/31eveningformula2}) and Lemma \ref{coefficients}, for any $(a,n,p)\in \Omega^{*}_{r,n_{1},p_{1},B_{2}}$, there is a positive integer $Z_{1}\ll O(1)$ such that for all $1\leq i\leq d'$,
\begin{equation}\label{10/31formula3}
\Bigg\|Z_{1}\Bigg(\sum_{w'=0}^{d'-i}\gamma_{w',i}a^{w'}+\sum_{w'=0}^{d'-i}\Big(\sum_{w\geq w'}^{d'-i}\beta_{w,i,w'}p^{w+i}\Big)a^{w'}\Bigg)\Bigg\|_{\mathbb{R}/\mathbb{Z}}\ll W^{O(B)}L^{-i}\ll 2^{ir}W^{O(B_{2})}H^{-i}.
\end{equation}
Using the pigeonhole principle, one can make $Z_{1}$ independent of $(a,n,p)$ after substituting $\Omega^{*}_{r,n_{1},p_{1},B_{2}}$ with a smaller subset whose cardinality still satisfies formula (\ref{12/06formula}). By formula (\ref{10/31formula4}), for any $p\in \mathcal{P}_{r,n_{1},p_{1}}$, there is a set $\mathcal{C}_{r,n_{1},p_{1},p}\subseteq [s]$ with
\begin{equation}\label{11/2formula1}
|\mathcal{C}_{r,n_{1},p_{1},p}|\gg W^{-O(B_{2})}s,
\end{equation}
such that for any $a\in \mathcal{C}_{r,n_{1},p_{1},p}$, $(a,p)$ satisfies formula (\ref{10/31formula3}). We now view \[Z_{1}\Bigg(\sum_{w'=0}^{d'-i}\gamma_{w',i}a^{w'}+\sum_{w'=0}^{d'-i}\Big(\sum_{w\geq w'}^{d'-i}\beta_{w,i,w'}p^{w+i}\Big)a^{w'}\Bigg)\] as a polynomial in $a$. Then applying Lemma \ref{CNnorm} (with $\epsilon=2^{ir}W^{O(B_{2})}H^{-i}$ and $\delta=W^{-O(B_{2})}$) to formula (\ref{10/31formula3}), there is a positive integer $Z_{2}$ with $Z_{2}\ll W^{O(B_{2})}$ such that for any $i=1,\ldots,d'$,
\[\Bigg\|Z_{2}Z_{1}\Bigg(\sum_{w'=0}^{d'-i}\gamma_{w',i}a^{w'}+\sum_{w'=0}^{d'-i}\Big(\sum_{w\geq w'}^{d'-i}\beta_{w,i,w'}p^{w+i}\Big)a^{w'}\Bigg)(mod~\mathbb{Z})\Bigg\|_{C^{\infty}[s]}\ll 2^{ir}W^{O(B_{2})}H^{-i}.\]
By Lemma \ref{coefficients} again, there is a positive integer $Z_{3}$ with $Z_{3}\ll O(1)$ such that for $i=1,\ldots,d'$ and $1\leq w'\leq d'-i$,
\begin{equation}\label{10/31formula5}
\Bigg\|Z_{3}Z_{2}Z_{1}\Big(\gamma_{w',i}+\sum_{w\geq w'}^{d'-i}\beta_{w,i,w'}p^{w+i}\Big)\Bigg\|_{\mathbb{R}/\mathbb{Z}}\ll  2^{ir}W^{O(B_{2})}H^{-i}s^{-w'}.
\end{equation}
Again by the pigeonhole principle, one can make $Z_{2}$ and $Z_{3}$ both independent of $p\in \mathcal{P}_{r,n_{1},p_{1}}$ after substituting $\mathcal{P}_{r,n_{1},p_{1}}$ with a smaller subset whose cardinality remains satisfying formula (\ref{10/31formula4}).
By formulas (\ref{10/31formula3}) and (\ref{10/31formula5}), and the triangle inequality, we have that for $i=1,\ldots,d'$ and $w'=0$,
\begin{equation}\label{11/2formula8}
\Bigg\|Z_{3}Z_{2}Z_{1}\Big(\gamma_{0,i}+\sum_{w\geq 0}^{d'-i}\beta_{w,i,0}p^{w+i}\Big)\Bigg\|_{\mathbb{R}/\mathbb{Z}}\ll  2^{ir}W^{O(B_{2})}H^{-i}.
\end{equation}
Applying Lemmas \ref{CNnorm} and \ref{coefficients} again to the following polynomials
\[Z_{3}Z_{2}Z_{1}\Big(\gamma_{w',i}+\sum_{w\geq w'}^{d'-i}\beta_{w,i,w'}p^{w+i}\Big)\]
of $p\in[2^{r}]$, then there are positive integers $Z_{4},Z_{5}$ with $Z_{4}\ll W^{O(B_{2})}$ and $Z_{5}\ll O(1)$, such that for any $i=1,\ldots,d'$, $0\leq w'\leq d'-i$ and $w'\leq w\leq d'-i$,
\begin{equation}\label{11/2formula3}
\|Z_{5}Z_{4}Z_{3}Z_{2}Z_{1}\beta_{w,i,w'}\|_{\mathbb{R}/\mathbb{Z}}\ll  2^{-wr}W^{O(B_{2})}H^{-i}s^{-w'}.
\end{equation}

Let $Z=Z_{5}Z_{4}Z_{3}Z_{2}Z_{1}$. Then $Z\ll W^{O(B_{2})}$ and the character $Z\eta_{2}$ satisfies
\begin{equation}\label{10/31formula6}
|Z\eta_{2}|\ll |Z||\eta|\ll W^{O(B_{2})}.
\end{equation}
Now we choose a sufficiently large positive number $C_{0}=O(1)$ which serves as the implicit constant in the exponent $O(B_{2})$ appearing in the above. We write formula (\ref{11/2formula3}) as
\begin{equation}\label{formula9}
\|Z\beta_{l_{1},l_{2},l_{1}'}\|_{\mathbb{R}/\mathbb{Z}}\ll 2^{-l_{1}r}W^{C_{0}B_{2}}H^{-l_{2}}s^{-l_{1}'},
\end{equation}
which holds for any $l_{1}\geq 0,l_{2}\geq 1$, $l_{1}+l_{2}\leq d'$ and $l_{1}'\leq l_{1}$.

By formula (\ref{11/7morning}), the number of pairs $(n,p)$ with $(n,p)\in \Omega_{r,n_1,p_1,B_2}^{*}$ is at most $2^{r}Hs$. Since $|\Omega_{r,n_1,p_1,B_2}^{*}|\geq 2^{r}W^{-C_{0}B_{2}}Hs^2$, there is a pair $(n,p)\in \mathcal{N}\times \{p:p\in (2^{r-1},2^{r}],(p,s)=1\}$ such that at least $W^{-C_{0}B_{2}}s$ choices of $a\in [s]$ have the property $(a,n,p)\in \Omega_{r,n_1,p_1,B_2}^{*}$. For such $(a,n,p)$, we consider the set $\mathcal{U}_{a}=\{u\in \mathbb{Z}:qsu+p(c_{0}+a)-n\in [Hs]\}$. By formula (\ref{10/31eveningformula8}), there is an $l_{0}\in [L]$ such that $qspl_{0}+p(c_{0}+a)-n\in [Hs]$. Since $L\leq 4\cdot2^{-r}W^{-3}H+4$, $-Hs< p(c_{0}+a)-n\leq Hs$. So for any $u\in \mathcal{U}_{a}$, $|u|\leq 2H/q$. Note that $\lfloor\frac{H}{q} \rfloor-2\leq |\mathcal{U}_{a}|\leq \lfloor\frac{H}{q} \rfloor+2$. Fix any subinterval $\mathcal{U}_{a}'\subseteq \mathcal{U}_{a}$ with length $\lfloor\frac{2W^{-2C_{0}B_{2}-3}H}{q}\rfloor$.  For any $u_{1},u_{2}\in \mathcal{U}_{a}'$, by formulas (\ref{10/31formula1}) and (\ref{formula9}), we have the following estimate.
\begin{align}\label{10/31formula7}
\begin{aligned}
&\|Z\eta_{2}\circ g'(n,qsu_{1}+p(c_{0}+a)-n)\Gamma'-Z\eta_{2}\circ g'(n,qsu_{2}+p(c_{0}+a)-n)\Gamma'\|_{\mathbb{R}/\mathbb{Z}}\\
=&\|Z\sum_{\substack{l_{1}\geq 0,l_{2}\geq 1,l_{1}'\geq 0\\l_{1}+l_{2}\leq d',l_{1}'\leq l_{1}}}\beta_{l_{1},l_{2},l_{1}'}p^{l_{1}}a^{l_{1}'}(u_{1}-u_{2})\sum_{h=0}^{l_{2}-1}u_{1}^{h}u_{2}^{l_{2}-1-h}\|_{\mathbb{R}/\mathbb{Z}}\\
\ll &\sum_{\substack{l_{1}\geq 0,l_{2}\geq 1,l_{1}'\geq 0\\l_{1}+l_{2}\leq d',l_{1}'\leq l_{1}}}W^{-C_{0}B_{2}-3}q^{-l_{2}}\leq C_{0} W^{-C_{0}B_{2}}.
\end{aligned}
\end{align}
Set $F(x)=e(Z\eta_{2}(x))$, which is a Lipschitz function from $G'/\Gamma'$ to $\mathbb{C}$ with $\|F\|_{Lip}\leq 2\pi W^{C_{0}B_{2}}$ by formula (\ref{10/31formula6}). Let $\mathcal{B}_{a}'=\{qsu+p(c_{0}+a)-n:u\in \mathcal{U}_{a}'\}$. Then by formula (\ref{10/31formula7}),
\[|\mathbb{E}_{h\in \mathcal{B}_{a}'}F(g'(n,h)\Gamma')|>1-C_{0}W^{-C_{0}B_{2}}> 1/2.\]

Note that $\mathcal{B}_{a}'$ is an arithmetic progression in $\mathcal{B}_{a}=\{h\in [Hs]:n+h\equiv pa(s)\}$ with length greater than $W^{-2C_{0}B_{2}-4}H$. By the above, it follows that the sequence $\{g'(n,h)\Gamma'\}_{h\in \mathcal{B}_{a}}$ is not totally $\min (W^{-2C_{0}B_{2}-4},\frac{1}{4\pi}W^{-C_{0}B_{2}})$-equidistributed in $G'/\Gamma'$. By the assumption in the proposition, $\min (W^{-2C_{0}B_{2}-4},\frac{1}{4\pi}W^{-C_{0}B_{2}})\geq W^{-C_{0}^{-1}B}$. So there is an $n\in \mathcal{N}$ such that for at least $W^{-C_{0}B_{2}}s$ choices of $a\in [s]$, $\{g'(n,h)\Gamma'\}_{h\in \mathcal{B}_{a}}$ is not totally $W^{-C_{0}^{-1}B}$-equidistributed in $G'/\Gamma'$. This contradicts the construction of $\mathcal{N}$ in formula (\ref{defofdensesetn}). We complete the proof.
\end{proof}

\end{document}